\documentclass{amsproc}%
\usepackage{amssymb}
\usepackage{amsmath}
\usepackage{amsfonts}
\usepackage{graphicx}%
\providecommand{\U}[1]{\protect\rule{.1in}{.1in}}
\newtheorem{theorem}{Theorem}[section]
\makeatletter
\@namedef{subjclassname@2010}{\textup{2010} Mathematics Subject Classification}
\makeatother

\theoremstyle{definition}

\newtheorem{proposition}{Proposition}
\newtheorem{corollary}{Corollary}
\theoremstyle{remark}
\newtheorem{remark}[theorem]{Remark}
\numberwithin{equation}{section}

\begin{document}
\title{Distance Functions for Reproducing Kernel Hilbert Spaces }
\author{N. Arcozzi}
\address{Dipartimento do Matematica, Universita di Bologna, 40127 Bologna, ITALY}
\email{arcozzi@dm.unibo.it }
\author{R. Rochberg}
\address{Department of Mathematics, Washington University, St. Louis, MO 63130, U.S.A}
\email{rr@math.wustl.edu}
\author{E. Sawyer}
\address{Department of Mathematics \& Statistics, McMaster University; Hamilton,
Ontairo, L8S 4K1, CANADA\ }
\email{sawyer@mcmaster.ca }
\author{B. D. Wick}
\address{School of Mathematics, Georgia Institute of Technology, 686 Cherry Street,
Atlanta, GA USA 30332--0160}
\email{wick@math.gatech.edu }
\thanks{The first author's work partially supported by the COFIN project Analisi
Armonica, funded by the Italian Minister for Research}
\thanks{The second author's work supported by the National Science Foundation under
Grant No. 0700238}
\thanks{The third author's work supported by the National Science and Engineering
Council of Canada.}
\thanks{The fourth author's work supported by the National Science Foundation under
Grant No. 1001098  and 0955432}
\date{*****}

\begin{abstract}
Suppose $H$ is a space of functions on $X$. If $H$ is a Hilbert space with
reproducing kernel then that structure of $H$ can be used to build distance
functions on $X$. We describe some of those and their interpretations and
interrelations. We also present some computational properties and examples.

\end{abstract}
\maketitle

\section{Introduction and Summary}

If $H$ is a Hilbert space with reproducing kernel then there is an associated
set, $X,$ and the elements of $H$ are realized as functions on $X. $ The space
$H$ can then be used to define distance functions on $X$. We will present
several of these and discuss their interpretations, interrelations and
properties. We find it particularly interesting that these ideas interface
with so many other areas of mathematics.

Some of our computations and comments are new but many of the details
presented here are known, although perhaps not as well known as they might be.
One of our goals in this note is to bring these details together and place
them in unified larger picture. The choices of specific topics however
reflects the recent interests of the authors and some relevant topics get
little or no mention.

The model cases for what we discuss are the hyperbolic and pseudohyperbolic
distance functions on the unit disk $\mathbb{D}$. We recall that material in
the next section. In the section after that we introduce definitions,
notation, and some basic properties of Hilbert spaces with reproducing
kernels. In Section 4 we introduce a function $\delta$, show that it is a
distance function on $X,$ and provide interpretations of it. In the section
after that we introduce a pair of distance functions first considered in this
context by Kobayshi and which, although not the same as $\delta,$ are closely
related. In Section 6 we discuss the relation between the distance functions
that have been introduced and distances coming from having a Riemannian metric
on $X.$ The model case for this is the relation between three quantities on
the disk, the pseudohyperbolic distance, its infinitesimal version, the
Poincare-Bergman metric tensor, and the associated geodesic distance, the
hyperbolic metric.

Several of the themes here are common in recent literature on reproducing
kernel Hilbert spaces. Some of the results here appear in the literature as
results for Bergman spaces, but in hindsight they extend almost without change
to larger classes of Hilbert space with reproducing kernel. Also, many results
for the Hardy space suggest natural and productive questions for reproducing
kernel Hilbert spaces with complete Nevanlinna Pick kernels. Those spaces have
substantial additional structure and $\delta$ then has additional
interpretations and properties. That is discussed in Section 7.

Section 8 focuses on computations of how $\delta$ changes when $H$ is replaced
by a subspace. Although that work is phrased in the language of distance
functions it can be seen as yet another instance of using extremal functions
for invariant or co-invariant subspaces as tools to study the subspaces.

Finally, we would like to emphasize that the material we present has not been
studied much and most of the natural questions one can ask in this area are open.

\section{Distances on the Unit Disk}

Here we collect some background material; references are \cite{G}, \cite{MPS},
and \cite{JP}.

The pseudohyperbolic metric, $\rho,$ is a metric on the unit disk,
$\mathbb{D}$, defined by, for $z,w\in\mathbb{D,}$%
\[
\rho\left(  z,w\right)  =\left\vert \frac{z-w}{1-\bar{z}w}\right\vert .
\]
Given any distance function $\sigma$ we can define the length of a curve
$\gamma:[a,b]\rightarrow\mathbb{D}$ by
\[
\ell_{\sigma}(\gamma)=\sup\left\{  \sum\limits_{i=o}^{n-1}\sigma(\gamma
(t_{i}),\gamma(t_{i+1})):a=t_{0}<t_{1}<...<t_{n}=b\right\}  .
\]
Using this functional we can define a new distance, $\sigma^{\ast},$ by
\[
\sigma^{\ast}\left(  z,w\right)  =\inf\left\{  \ell_{\sigma}(\gamma
):\gamma\text{ a curve joining }z\text{ to }w\right\}  .
\]
Automatically $\sigma^{\ast}\geq\sigma$ and if equality holds $\sigma$ is
called an \textit{inner distance. }More generally $\sigma^{\ast}$ is referred
to as the inner distance generated by $\sigma.$

The distance $\rho$ is not an inner distance. The associated $\rho^{\ast}$ is
the hyperbolic distance, $\beta,$ which is related to $\rho$ by%
\[
\beta=\log\frac{1+\rho}{1-\rho},\text{ }\rho=\frac{1}{2}\tanh\beta.
\]
The hyperbolic distance can also be obtained by regarding the disk as a
Riemannian manifold with length element%
\[
ds=\frac{2\left\vert dz\right\vert }{1-\left\vert z\right\vert ^{2}}%
\]
in which case $\beta(z,w)$ is the length of the geodesic connecting $z$ to
$w.$

The Hardy space, $H^{2}=H^{2}\left(  \mathbb{D}\right)  ,$ is the Hilbert
space of functions, $f(z)=\sum_{n=0}^{\infty}a_{n}z^{n}$ which are holomorphic
on the disk and for which $\left\Vert f\right\Vert ^{2}=\sum\left\vert
a_{n}\right\vert ^{2}<\infty.$ The inner product of $f$ with $g(z)=\sum
b_{n}z^{n}$ is \thinspace$\left\langle f,g\right\rangle =\sum a_{n}\bar{b}%
_{n}.$ The Hardy space is a Hilbert space with reproducing kernel. That is,
for each $\zeta\in\mathbb{D}$ there is a kernel function $k_{\zeta}\in H^{2}$
which reproduces the value of functions at $\zeta;$ $\forall f\in H^{2},$
$\left\langle f,k_{\zeta}\right\rangle =f(\zeta).$ It is straightforward to
see that there is at most one such function and that $k_{\zeta}\left(
z\right)  =\left(  1-\bar{\zeta}z\right)  ^{-1}$ has the required property.

For the Hardy space now, and later for a general reproducing kernel Hilbert
space, we are interested in the functional $\delta\left(  \cdot,\cdot\right)
, $ defined for $\left(  z,w\right)  \in$ $\mathbb{D\times D}$, by
\begin{equation}
\delta(z,w)=\delta_{H^{2}}(z,w)=\sqrt{1-\left\vert \left\langle \frac{k_{z}%
}{\left\Vert k_{z}\right\Vert },\frac{k_{w}}{\left\Vert k_{w}\right\Vert
}\right\rangle \right\vert ^{2}}. \label{Hardy delta}%
\end{equation}
For the Hardy space this is evaluated as%
\begin{equation}
\delta_{H^{2}}(z,w)=\sqrt{1-\frac{(1-\left\vert z\right\vert ^{2})(1-|w|^{2}%
)}{\left\vert 1-\bar{z}w\right\vert ^{2}}}.
\end{equation}
This can be simplified using a wonderful identity. For $z,w\in\mathbb{D}$%
\begin{equation}
1-\frac{(1-\left\vert z\right\vert ^{2})(1-|w|^{2})}{\left\vert 1-\bar
{z}w\right\vert ^{2}}=\left\vert \frac{z-w}{1-\bar{z}w}\right\vert ^{2}.
\label{magic}%
\end{equation}
Hence $\delta_{H^{2}}(z,w)=\rho(z,w).$

\section{Reproducing Kernel Hilbert Spaces}

By a reproducing kernel Hilbert space, RKHS, we mean a Hilbert space $H$ of
functions defined on a set $X$ together with a function $K(\cdot,\cdot)$
defined on $X\times X$ with two properties; first, $\forall x\in X,$
$k_{x}(\cdot)=K(\cdot,x)\in H,$ second $\forall f\in H$ $\left\langle
f,k_{x}\right\rangle =f(x).$ The function $k_{x}$ is called the reproducing
kernel for the point $x.$ We will use the following notation for unit vectors
in the direction of the kernel functions. For $x\in X$ we set
\[
\hat{k}_{x}=\frac{k_{x}}{\left\Vert k_{x}\right\Vert }.
\]

General background on such spaces can be found, for instance, in \cite{AM}.
Here we will just mention three families of examples and collect some standard facts.

\subsection{Examples}

\begin{center}
\textbf{The Dirichlet-Hardy-Bergman Family}
\end{center}

For $\alpha>0$ let $\mathcal{H}_{\alpha}$ be the RKHS of holomorphic functions
on $\mathbb{D}$ with the reproducing kernel
\[
K_{\alpha}\left(  w,z\right)  =k_{\alpha,z}(w)=\left(  1-\bar{z}w\right)
^{-\alpha}.
\]
For $\alpha=0$ there is the limit version
\[
K_{0}\left(  w,z\right)  =k_{0,z}(w)=\frac{1}{\bar{z}w}\log\frac{1}{1-\bar
{z}w}.
\]
We have not included normalizing multiplicative constants as part of the
kernels; we will be dealing only with expressions similar to
(\ref{Hardy delta}) which are not affected by such constants. Also, we have
not specified the norms for the spaces. In fact we will need to know how to
take inner products with kernel functions but we will never need exact
formulas for norms of general functions in the spaces. Hence we will give
Hilbert space norms for the function spaces which are equivalent to the
intrinsic RKHS norms.

First we consider the case $\alpha>1.$ These are generalized Bergman spaces;
$f\in\mathcal{H}_{\alpha}$ if and only if
\[
\left\Vert f\right\Vert _{\mathcal{H}_{\alpha}}^{2}\sim\int\int_{\mathbb{D}%
}\left\vert f(z)\right\vert ^{2}(1-\left\vert z\right\vert ^{2})^{\alpha
-2}dxdw<\infty.
\]
The case $\alpha=2$ is the classical Bergman space.

If $\alpha=1$ we have the Hardy space described earlier. In that case the norm
can be given using the radial boundary values $f^{\ast}(e^{i\theta})$ by
\[
\left\Vert f\right\Vert _{\mathcal{H}_{0}}^{2}=\int_{\partial\mathbb{D}%
}\left\vert f^{\ast}(e^{i\theta})\right\vert ^{2}d\theta.
\]
An equivalent norm for the Hardy space is%

\[
\left\Vert f\right\Vert _{\mathcal{H}_{1}}^{2}\sim\left\vert f(0)\right\vert
^{2}+\int\int_{\mathbb{D}}\left\vert f^{\prime}(z)\right\vert ^{2}%
(1-\left\vert z\right\vert ^{2})dxdw.
\]
The second description of the norm for the Hardy space is the one which
generalizes to $\alpha<1$.  For $0\leq\alpha\leq1$, $f$ is in $\mathcal{H}%
_{\alpha}$ exactly if
\[
\left\Vert f\right\Vert _{\mathcal{H}_{\alpha}}^{2}\sim\left\vert
f(0)\right\vert ^{2}+\int\int_{\mathbb{D}}\left\vert f^{\prime}(z)\right\vert
^{2}(1-\left\vert z\right\vert ^{2})^{\alpha}dxdw<\infty.
\]
The space $\mathcal{H}_{0}$ is the Dirichlet space and the $\mathcal{H}%
_{\alpha}$ for $0<\alpha<1$ are called generalized Dirichlet spaces.

\begin{center}
\textbf{The Fock-Segal-Bargmann Scale}
\end{center}

For $\beta>0$ let $F_{\beta}$ be the Hilbert space of holomorphic functions on
$\mathbb{C}$ for which
\[
\left\Vert f\right\Vert _{F_{\beta}}^{2}\sim\int\int_{\mathbb{C}}\left\vert
f(z)\right\vert ^{2}e^{-\beta\left\vert z\right\vert ^{2}}dxdy<\infty.
\]
This is a RKHS and the kernel function is given by%
\[
K_{\beta}\left(  z,w\right)  =e^{\beta zw}.
\]

\begin{remark}
There are other families of RKHS for which the kernel functions are powers of
each other and still others where such relations hold asymptotically, see, for
instance \cite{E}, \cite{JPR}.
\end{remark}

\begin{center}
\textbf{General Bergman Spaces}
\end{center}

Suppose $\Omega$ is a bounded domain in $\mathbb{C}$ or, for that matter,
$\mathbb{C}^{n}.$ The Bergman space of $\Omega$, $B(\Omega),$ is the space of
all functions holomorphic on $\Omega$ which are square integrable with respect
to volume measure; $f\in B(\Omega)$ exactly if
\[
\left\Vert f\right\Vert _{B(\Omega)}^{2}=\int\int_{\Omega}\left\vert
f(z)\right\vert ^{2}dV(z)<\infty.
\]
In this case it is easy to see that $B(\Omega)$ is a Hilbert space and that
evaluation at points of $\Omega$ are continuous functionals and hence are
given by inner products with some kernel functions. However, and this is one
of the reasons for mentioning this example, it is generically not possible to
write explicit formulas for the kernel functions.

\subsection{Multipliers}

Associated with a RKHS $H$ is the space $M(H)$ of multipliers of $H,$
functions $m$ defined on $X$ with the property that multiplication by $m$ is a
bounded map of $H$ into itself. For $m\in M(H)$ we will denote the operator of
multiplication by $m$ by $M_{m}.$ The \textit{multiplier norm }of $m$ is
defined to be the operator norm of $M_{m}.$

For example, the multiplier algebra $M(B(\Omega))$ consists of all bounded
analytic functions on $\Omega.$ The multiplier algebra of any of the spaces
$F_{\beta}$ consists of only the constant functions.

\subsection{ Background Facts}

Suppose that $H$ is a RKHS on $X$ with kernel functions $\left\{
k_{x}\right\}  _{x\in X}$ and multiplier algebra $M(H).$ The following are
elementary Hilbert space results.

\begin{proposition}
Suppose $f\in H,$ $\left\Vert f\right\Vert \leq1,x\in H$. The maximum possible
value of $\operatorname{Re}f(z)$ (and hence also of $\left\vert
f(z)\right\vert )$ is the value $\left\Vert k_{z}\right\Vert =k_{z}(z)^{1/2}$
attained by the unique function $f=\hat{k}_{z}.$
\end{proposition}

\begin{proposition}
\label{extremal}There is a unique $F_{z,w}\in H$ with $\left\Vert
F_{z,w}\right\Vert _{H}\leq1$ and $F_{z,w}(z)=0$ which maximizes
$\operatorname{Re}F_{z,w}(w).$ It is given by
\[
F_{z,w}\left(  \cdot\right)  =\frac{k_{w}\left(  \cdot\right)  -\left\Vert
k_{z}\right\Vert ^{-2}k_{w}\left(  z\right)  k_{z}\left(  \cdot\right)
}{\left\Vert k_{w}\right\Vert \sqrt{1-\left\vert \left\langle \hat{k}_{z}%
,\hat{k}_{w}\right\rangle \right\vert ^{2}}.}%
\]
and it has
\[
F_{z,w}\left(  w\right)  =\left\Vert k_{w}\right\Vert \sqrt{1-\left\vert
\left\langle \hat{k}_{z},\hat{k}_{w}\right\rangle \right\vert ^{2}}.
\]

\end{proposition}

\begin{proposition}
\label{multiplier}For $m\in M(H), x\in X$ we have $\left(  M_{m}\right)
^{\ast}k_{x}=\overline{m\left(  x\right)  }k_{x}$
\end{proposition}

\begin{proposition}
\label{multiplier extremal}Suppose $m\in M(H), x,y\in X.$ If $\left\Vert
M_{m}\right\Vert _{M(H)}\leq1$ and $m(x)=0$ then
\end{proposition}

\[
\left\vert m(y)\right\vert \leq\sqrt{1-\left\vert \left\langle \hat{k}%
_{x},\hat{k}_{y}\right\rangle \right\vert ^{2}}.
\]

\section{The Sine of the Angle}

Suppose we have a RKHS $H$ of functions on $X$ and we want to introduce a
metric on $X$ that reflects the properties of functions in $H.$ There are
various ways to do this, for instance we could declare the distance between
$x,y\in X$ to be $\left\Vert \hat{k}_{x}-\hat{k}_{y}\right\Vert .$ Here we
focus on a different choice. Motivated by, among other things,\thinspace the
modulus of continuity estimates in Proposition \ref{extremal} and Proposition
\ref{multiplier extremal} we define, if neither $k_{x}$ nor $k_{y}$ is the
zero function,%
\begin{equation}
\delta(x,y)=\delta_{H}(x,y)=\sqrt{1-\left\vert \left\langle \hat{k}_{x}%
,\hat{k}_{y}\right\rangle \right\vert ^{2}}. \label{def1}%
\end{equation}
We don't have a satisfactory definition of $\delta(x,y)$ if $k_{x}$ or $k_{y}
$ is the zero function. Either declaring these distances to be $1$ or to be
$0$ would lead to awkwardness later. Instead we leave $\delta$ undefined in
such cases. However we will overlook that fact and, for instance write%
\[
\forall x,y\in X,\delta_{H_{1}}(x,y)=\delta_{H_{2}}(x,y)
\]
to actually mean that the stated equality holds for all $x,y$ for which both
sides are defined.

One way to interpret $\delta$ is to note that, by virtue of the propositions
in the previous section,%
\[
\delta_{H}(x,y)=\frac{\sup\left\{  \left\vert f(y)\right\vert :f\in
H,\left\Vert f\right\Vert =1,f(x)=0\right\}  }{\sup\left\{  \left\vert
f(y)\right\vert :f\in H,\left\Vert f\right\Vert =1\right\}  }.
\]
Also, $\delta(x,y)$ measures how close the unit vectors $\hat{k}_{x}$ and
$\hat{k}_{y}$ are to being parallel. If $\theta$ is the angle between the two
then $\delta(x,y)=\sqrt{1-\cos^{2}\theta}=\left\vert \sin\theta\right\vert .$

In fact $\delta$ is a pseudo-metric. It is clearly symmetric. It is positive
semidefinite and will be positive definite if $H$ separates points of $X.$
(Although we will consider spaces which do not separate all pairs of points we
will still refer to $\delta$ as a metric.) The triangle inequality can be
verified by a simple argument \cite[Pg. 128]{AM}. Instead we proceed to
computations which develop further the idea that $\delta$ measures the
distance between points in the context of $H.$ A corollary of the first of
those computations is that $\delta$ satisfies the triangle inequality.

For a linear operator $L$ we denote the operator norm by $\left\Vert
L\right\Vert $ and the trace class norm by $\left\Vert L\right\Vert
_{\operatorname*{Trace}}.$ If $L$ is a rank $n$ self adjoint operator then it
will have real eigenvalues $\left\{  \lambda_{i}\right\}  _{i=1}^{n}.$ In that
case we have
\[
\left\Vert L\right\Vert =\sup\left\{  \left\vert \lambda_{i}\right\vert
\right\}  ,\text{ }\left\Vert L\right\Vert _{\operatorname*{Trace}}%
=\sum\left\vert \lambda_{i}\right\vert ,\text{ }\operatorname*{Trace}\left(
L\right)  =\sum\lambda_{i}.
\]
Also, recall that if $L$ is acting on a finite dimensional space then
$\operatorname*{Trace}\left(  L\right)  $ equals the sum of the diagonal
elements of any matrix which represents $L$ with respect to an orthonormal basis.

\begin{proposition}
[Coburn \cite{CO2}]\label{norm}For $x,y\in X$ let $P_{x}$ and $P_{y}$ be the
self adjoint projections onto the span of $k_{x}$ and $k_{y}$ respectively.
With this notation
\[
\delta(x,y)=\left\Vert P_{x}-P_{y}\right\Vert =\frac{1}{2}\left\Vert
P_{x}-P_{y}\right\Vert _{\operatorname*{Trace}}.
\]

\begin{proof}
$P_{x}$ and $P_{y}$ are rank one self adjoint projections and hence have trace
one. Thus the difference, $P_{x}-P_{y},$ is a rank two self adjoint operator
with trace zero and so it has two eigenvalues, $\pm\lambda$ for some
$\lambda\geq0.$ Thus $\left\Vert P_{x}-P_{y}\right\Vert =\lambda,$ $\left\Vert
P_{x}-P_{y}\right\Vert _{\operatorname*{Trace}}=2\lambda.$ We will be finished
if we show $2\delta(x,y)^{2}=2\lambda^{2}.$ We compute
\begin{align*}
2\lambda^{2}  &  =\operatorname*{Trace}\left(  \left(  P_{x}-P_{y}\right)
^{2}\right) \\
&  =\operatorname*{Trace}\left(  P_{x}+P_{y}-P_{x}P_{y}-P_{y}P_{x}\right) \\
&  =2-\operatorname*{Trace}\left(  P_{x}P_{y}\right)  -\operatorname*{Trace}%
\left(  P_{y}P_{x}\right) \\
&  =2-2\operatorname*{Trace}\left(  P_{x}P_{y}\right)  .
\end{align*}
Going to last line we used the fact that $\operatorname*{Trace}\left(
AB\right)  =\operatorname*{Trace}\left(  BA\right)  $ for any $A,B.$ We now
compute $\operatorname*{Trace}\left(  P_{x}P_{y}\right)  .$ Let $V$ be the
span of $k_{x}$ and $k_{y}.$ $P_{x}P_{y}$ maps $V$ into itself and is
identically zero on $V^{\perp}$ hence we can evaluate the trace by regarding
$P_{x}P_{y}$ as an operator on $V,$ picking an orthonormal basis for $V,$ and
summing the diagonal elements of the matrix representation of $V$ with respect
to that basis. We select the basis $\hat{k}_{y}$ and $j$ where $j$ is any unit
vector in $V$ orthogonal to $k_{y}.$ Noting that $P_{y}\hat{k}_{y}=\hat{k}%
_{y}$ and $P_{y}j=0$ we compute
\begin{align*}
\operatorname*{Trace}\left(  P_{x}P_{y}\right)   &  =\left\langle P_{x}%
P_{y}\hat{k}_{y},\hat{k}_{y}\right\rangle +\left\langle P_{x}P_{y}%
j,j\right\rangle \\
&  =\left\langle P_{x}\hat{k}_{y},\hat{k}_{y}\right\rangle +0\\
&  =\left\langle \left\langle \hat{k}_{y},\hat{k}_{x}\right\rangle \hat{k}%
_{x},\hat{k}_{y}\right\rangle \\
&  =\left\vert \left\langle \hat{k}_{y},\hat{k}_{x}\right\rangle \right\vert
^{2}%
\end{align*}
which is what we needed.
\end{proof}
\end{proposition}

\begin{remark}
Because we actually found the eigenvalues of $P_{x}$ and $P_{y}$ we can also
write $\delta$ in terms of any of the Schatten $p$-norms, $1\leq p<\infty;$
$\delta(x,y)=2^{-1/p}\left\Vert P_{x}-P_{y}\right\Vert _{S_{p}}.$
\end{remark}

A similar type of computation allows us to compute the operator norm of the
commutator $\left[  P_{a},P_{b}\right]  =P_{a}P_{b}-P_{b}P_{a}.$ Informally,
if $a$ and $b$ are very far apart, $\delta(a,b)\sim1,$ then each of the two
products will be small and hence so will the commutator. If the points are
very close, $\delta(a,b)\sim0,$ the individual products will be of moderate
size and almost equal so their difference will be small.

\begin{proposition}
$\left\Vert \left[  P_{a},P_{b}\right]  \right\Vert ^{2}=\delta(a,b)^{2}%
\left(  1-\delta(a,b)^{2}\right)  .$

\begin{proof}
Note that $\left[  P_{a},P_{b}\right]  $ is a skew adjoint rank two operator
of trace $0$ and hence has eigenvalues $\pm i\lambda,$ for some $\lambda>0.$
Hence $\left\Vert \left[  P_{a},P_{b}\right]  \right\Vert =\lambda.$ Also,
$\left[  P_{a},P_{b}\right]  ^{\ast}\left[  P_{a},P_{b}\right]  $ is a
positive rank two operator with eigenvalues $\lambda^{2},$ $\lambda^{2}$ so
its trace is $2\lambda^{2}.$ We now compute%
\begin{align*}
\left[  P_{a},P_{b}\right]  ^{\ast}\left[  P_{a},P_{b}\right]   &  =P_{a}%
P_{b}P_{b}P_{a}-P_{a}P_{b}P_{a}P_{b}-P_{b}P_{a}P_{b}P_{a}+P_{b}P_{a}P_{a}%
P_{b}\\
&  =P_{a}P_{b}P_{a}-P_{a}P_{b}P_{a}P_{b}-P_{b}P_{a}P_{b}P_{a}+P_{b}P_{a}P_{b}%
\end{align*}
Recalling that for any $A,B,\operatorname*{Trace}\left(  AB\right)
=\operatorname*{Trace}\left(  BA\right)  $ and also that the projections are
idempotent we see that
\[
\operatorname*{Trace}\left(  P_{a}P_{b}P_{a}\right)  =\operatorname*{Trace}%
\left(  P_{a}P_{b}\right)  =\operatorname*{Trace}\left(  P_{b}P_{a}\right)
=\operatorname*{Trace}\left(  P_{b}P_{a}P_{b}\right)  .
\]
The two middle quantities were computed in the previous proof
\[
\operatorname*{Trace}\left(  P_{a}P_{b}\right)  =\operatorname*{Trace}\left(
P_{b}P_{a}\right)  =\left\vert \left\langle \hat{k}_{a},\hat{k}_{b}%
\right\rangle \right\vert ^{2}.
\]
We also have
\[
\operatorname*{Trace}(P_{a}P_{b}P_{a}P_{b})=\operatorname*{Trace}\left(
P_{b}P_{a}P_{b}P_{a}\right)
\]
We compute the trace of the rank two operator $P_{b}P_{a}P_{b}P_{a}$ by
summing the diagonal entries of the matrix representation of the operator with
respect to an orthonormal basis consisting of $\hat{k}_{a}$ and $j,$ a unit
vector orthogonal to $\hat{k}_{a}.$%
\begin{align}
\operatorname*{Trace}\left(  P_{b}P_{a}P_{b}P_{a}\right)   &  =\left\langle
P_{b}P_{a}P_{b}P_{a}\hat{k}_{a},\hat{k}_{a}\right\rangle +\left\langle
P_{b}P_{a}P_{b}P_{a}j,j\right\rangle \nonumber\\
&  =\left\langle P_{b}P_{a}P_{b}\hat{k}_{a},\hat{k}_{a}\right\rangle +0
\label{trace}%
\end{align}
Next note that
\begin{align*}
P_{b}P_{a}P_{b}\hat{k}_{a}  &  =\left\langle \hat{k}_{a},\hat{k}%
_{b}\right\rangle P_{b}P_{a}\hat{k}_{b}\\
&  =\left\langle \hat{k}_{a},\hat{k}_{b}\right\rangle \left\langle \hat{k}%
_{b},\hat{k}_{a}\right\rangle P_{b}\hat{k}_{a}\\
&  =\left\langle \hat{k}_{a},\hat{k}_{b}\right\rangle \left\langle \hat{k}%
_{b},\hat{k}_{a}\right\rangle \left\langle \hat{k}_{a},\hat{k}_{b}%
\right\rangle \hat{k}_{b}%
\end{align*}
and hence we can evaluate (\ref{trace}) and obtain $\left\vert \left\langle
\hat{k}_{a},\hat{k}_{b}\right\rangle \right\vert ^{4}.$ Thus
\begin{align*}
\left\Vert \left[  P_{a},P_{b}\right]  \right\Vert ^{2}  &  =\frac{1}%
{2}\operatorname*{Trace}\left(  \left[  P_{a},P_{b}\right]  ^{\ast}\left[
P_{a},P_{b}\right]  \right) \\
&  =\frac{1}{2}\left(  2\operatorname*{Trace}\left(  P_{a}P_{b}\right)
-2\operatorname*{Trace}\left(  P_{b}P_{a}P_{b}P_{a}\right)  \right) \\
&  =\left\vert \left\langle \hat{k}_{a},\hat{k}_{b}\right\rangle \right\vert
^{2}-\left\vert \left\langle \hat{k}_{a},\hat{k}_{b}\right\rangle \right\vert
^{4}=\left(  1-\left\vert \left\langle \hat{k}_{a},\hat{k}_{b}\right\rangle
\right\vert ^{2}\right)  \left\vert \left\langle \hat{k}_{a},\hat{k}%
_{b}\right\rangle \right\vert ^{2}\\
&  =\delta(a,b)^{2}\left(  1-\delta(a,b)^{2}\right)  .
\end{align*}

\end{proof}
\end{proposition}

\noindent\textbf{Hankel forms: }The projection operators of the previous
section can be written as $P_{a}=\hat{k}_{a}\rangle\langle\hat{k}_{a}.$ In
some contexts these are the primordial normalized Toeplitz operators. More
precisely, on a Bergman space these are the Toeplitz operators with symbol
given by the measure $\left\Vert k_{a}\right\Vert ^{-2}\delta_{a}$ and general
Toeplitz operators are obtained by integrating fields of these. There is also
a map which takes functions on $X$ to bilinear forms on $H\times H.$ The basic
building blocks for that construction are the norm one, rank one (small)
Hankel forms given by $L_{a}=\hat{k}_{a}\otimes\hat{k}_{a},$ thus
\[
L_{a}\left(  f,g\right)  =\left\langle f,\hat{k}_{a}\right\rangle \left\langle
g,\hat{k}_{a}\right\rangle =f(a)g(a).
\]
Limits of sums of these, or, equivalently, integrals of fields of these; are
the Hankel forms on $H;$ for more on this see \cite{JPR}.

The norm of a bilinear form $B$ on $H\times H$ is
\[
\left\Vert B\right\Vert =\sup\left\{  \left\vert B(f,g)\right\vert :f,g\in
H,\left\Vert f\right\Vert =\left\Vert g\right\Vert =1\right\}  .
\]
Associated to a bounded $B$ is a bounded conjugate linear map $\beta$ of $H$
to itself defined by $\left\langle f,\beta g\right\rangle =B(f,g).$ If we then
define a conjugate linear $\beta^{\ast}$ by $\left\langle \beta^{\ast
}f,g\right\rangle =\left\langle \beta g,f\right\rangle $ then
\[
\left\langle \beta^{\ast}\beta f,f\right\rangle =\left\langle \beta f,\beta
f\right\rangle \geq0.
\]
Thus $\beta^{\ast}\beta$ is a positive linear operator. The form norm of $B$
equals the operator norm of $\left(  \beta^{\ast}\beta\right)  ^{1/2}$ and we
define the trace class norm of $B$ to be the trace of the positive operator
$\left(  \beta^{\ast}\beta\right)  ^{1/2}.$ With these definitions in hand we
have a complete analog of Proposition \ref{norm}.

\begin{proposition}
For $x,y\in X$%
\[
\delta(x,y)=\left\Vert L_{x}-L_{y}\right\Vert =\frac{1}{2}\left\Vert
L_{x}-L_{y}\right\Vert _{\operatorname*{Trace}}.
\]

\begin{proof}
Let $\beta_{x}$ and $\beta_{y}$ be the conjugate linear maps associated with
$L_{x}$ and $L_{y}.$ One computes that $\beta_{x}f=\beta_{x}^{\ast
}f=\left\langle \hat{k}_{a},f\right\rangle \hat{k}_{a}$ and similarly for
$\beta_{y}.$ Using this one then checks that for any $x,y;$ $\beta_{y}^{\ast
}\beta_{x}=P_{y}P_{x}.$ Thus the proof of Proposition \ref{norm} goes through.
\end{proof}
\end{proposition}

\section{Formal Properties}

We collect some observations on how the metric $\delta$ interacts with some
basic constructions on RKHS's.

\subsection{Direct Sums}

If $H$ is a RKHS of functions on a set $X$ and $J$ is a RKHS on a disjoint set
$Y$ then we can define a RKHS $(H,J)$ on $X\cup Y$ to be the space of pairs
$\left(  h,j\right)  $ with $h\in H,j\in J$ regarded as functions on $X\cup Y$
via the prescription%
\[
(h,j)(z)=%
\genfrac{\{}{.}{0pt}{}{h(z)\text{ if }z\in X}{j(z)\text{ if }z\in Y.}%
\]
One then easily checks that%
\[
\delta_{(H,J)}(z,z^{\prime})=\left\{
\begin{array}
[c]{cc}%
\delta_{H}(z,z^{\prime}) & \text{if }z,z^{\prime}\in X\\
\delta_{J}(z,z^{\prime}) & \text{if }z,z^{\prime}\in Y\\
1 & \text{otherwise}%
\end{array}
\right.  .
\]

That computation is fundamentally tied to the fact that $H$ and $J$ are sets
of functions on \textit{different }spaces. If, however, all the spaces
considered are functions on the \textit{same }space then it is not clear what
the general pattern is. That is, if $H$ is a RKHS on $X$ and if $J,J^{\prime}$
are two closed subspaces of $H$ with, hoping to simplify the situation,
$J\perp J^{\prime}$ then there seems to be no simple description of the
relationship between $\delta_{H},\delta_{J},\delta_{J^{\prime}},$ and
$\delta_{J\oplus J^{\prime}}.$ In some of the examples in a later section we
will compute these quantities with $J^{\prime}=J^{\perp}$ but no general
statements are apparent.

\subsection{Rescaling}

Suppose $H$ is a RKHS of functions on $X$ and suppose that $G(x)$ is a
nonvanishing function on $X;$ $G$ need not be in $H$ and it need not be
bounded. The associated rescaled space, $GH,$ is the space of functions
$\left\{  Gh:h\in H\right\}  $ with the inner product%
\[
\left\langle Gf,Gg\right\rangle _{GH}=\left\langle f,g\right\rangle _{H}.
\]
It is straightforward to check that $GH$ is an RKHS and that its kernel
function, $K_{GH}$ is related to that of $H,$ $K_{H}$ by%
\[
K_{GH}\left(  x,y\right)  =G(x)\overline{G(y)}K_{H}(x,y).
\]
An immediate consequence of this is that $\delta$ does not see the change;
$\delta_{GH}=\delta_{H}.$

Elementary examples of rescaling show that certain types of information are
not visible to $\delta.$ Suppose we rescale a space $H$ to the space $cH$ for
a number $c$ (that is; $\left\vert c\right\vert ^{2}\left\langle
f,g\right\rangle _{cH}=\left\langle f,g\right\rangle _{H}$). The natural
``identity'' map from $H$ to $cH$ which takes the function $f$ to the function
$f$ will be, depending on the size of $c,$ a strict expansion of norms, an
isometry, or a strict contraction. However it is not clear how one can
recognize these distinctions by working with $\delta_{H}$ and $\delta_{cH}.$

An awkward fact about rescaling is that sometimes it is present but not
obviously so. Consider the following two pair of examples. First, let $H$ be
$\mathcal{H}_{1},$ the Hardy space of the disk. This can be realized as the
closure of the polynomials with respect to the norm%
\[
\left\Vert \sum a_{k}z^{k}\right\Vert _{\mathcal{H}_{1}}^{2}=\int_{0}^{2\pi
}\left\vert \sum a_{k}e^{ik\theta}\right\vert ^{2}\frac{d\theta}{2\pi}.
\]
For a weight, a smooth positive function $w(\theta)$ defined on the circle,
let $\mathcal{H}_{1,w}$ be the weighted Hardy space; the space obtained by
closing the polynomials using the norm
\[
\left\Vert \sum a_{k}z^{k}\right\Vert _{\mathcal{H}_{1,w}}^{2}=\int_{0}^{2\pi
}\left\vert \sum a_{k}e^{ik\theta}\right\vert ^{2}w(\theta)\frac{d\theta}%
{2\pi}.
\]
This is also a RKHS on the disk and in fact these two spaces are rescalings of
each other. However to see that one needs to use a bit of function theory. The
crucial fact is that one can write $w(\theta)$ as
\[
w(\theta)=\left\vert W\left(  e^{i\theta}\right)  \right\vert ^{2}%
\]
with $W(z)$ and $1/W(z)$ holomorphic in the disk and having continuous
extensions to the closed disk. The functions $W^{\pm1}$ can then be used to
construct the rescalings.

We now do a similar construction for the Bergman space. That space,
$\mathcal{H}_{2}$ in our earlier notation, is the space of holomorphic
functions on the disk normed by%
\[
\left\Vert f\right\Vert _{\mathcal{H}_{2}}^{2}=\int\int_{\mathbb{D}}\left\vert
f(z)\right\vert ^{2}dxdy.
\]
A weighted version of this space is given by replacing $dxdy$ by $w(z)dxdy$
for some smooth positive bounded $w.$ To make the example computationally
tractable we suppose $w$ is radial; $w(z)=v(\left\vert z\right\vert ).$ We
define the weighted Bergman space $\mathcal{H}_{2,w}$ by the norming function
\[
\left\Vert f\right\Vert _{\mathcal{H}_{2,w}}^{2}=\int\int_{\mathbb{D}%
}\left\vert f(z)\right\vert ^{2}w(z)dxdy.
\]
The space $\mathcal{H}_{2,w}$ is a RKHS on the disk and is an equivalent
renorming of $\mathcal{H}_{2}$ but is not related by rescaling. One way to see
this is to note that, because the densities $1$ and $w(z)$ are both radial, in
both cases the monomials are a complete orthogonal set. Thus, in both cases,
the kernel function restricted to the diagonal is of the form%
\[
K(z,z)=\sum_{0}^{\infty}\frac{\left\vert z\right\vert ^{2n}}{\left\Vert
z^{n}\right\Vert ^{2}}=a_{0}+a_{1}\left\vert z\right\vert ^{2}+\cdot\cdot
\cdot.
\]
Hence we can compute that for $z$ near the origin%
\[
\delta(0,z)=\frac{\left\Vert 1\right\Vert }{\left\Vert z\right\Vert
}\left\vert z\right\vert (1+O(\left\vert z\right\vert ^{2})).
\]
If the spaces were rescalings of each other then the coefficients $\left\Vert
1\right\Vert /\left\Vert z\right\Vert $ would have to match, but this is not
true independently of the choice of $w$.

\subsection{Products of Kernels}

In some cases the kernel function for a RKHS has a product structure. We begin
by recalling two constructions that lead to that situation. Suppose that for
$i=1,2;$ $H_{i}$ is a RKHS on $X_{i}.$ We can regard the Hilbert space tensor
product $H_{1}\otimes H_{2}$ as a space of functions on the product
$X_{1}\times X_{2}$ by identifying the elementary tensor $h_{1}\otimes h_{2}$
with the function on $X_{1}\times X_{2}$ whose value at $\left(  x_{1}%
,x_{2}\right)  $ is $h_{1}(x_{1})h_{2}(x_2).$ It is a standard fact that this
identification gives $H_{1}\otimes H_{2}$ the structure of a RKHS on
$X_{1}\times X_{2}$ and, denoting the three kernel functions by $K_{1},$
$K_{2},$ and $K_{1,2}$ we have
\[
K_{1,2}\left(  \left(  x_{1},x_{2}\right)  ,\left(  x_{1}^{\prime}%
,x_{2}^{\prime}\right)  \right)  =K_{1}\left(  x_{1},x_{1}^{\prime}\right)
K_{2}\left(  x_{2},x_{2}^{\prime}\right)  .
\]
\ Now suppose further that $X_{1}=X_{2}$ and denote both by $X.$ The mapping
of$\ x\in X$ to $\left(  x,x\right)  \in X\times X$ lets us identify $X$ with
the diagonal $D\subset X\times X$ and we will use this identification to
describe a new RKHS, $H_{12},$ of functions on $X$ (now identified with $D).$
The functions in $H_{12}$ are exactly the functions obtained by restricting
elements of $H_{1}\otimes H_{2}$ to $D.$ The Hilbert space structure is given
as follows. For $f,g\in H_{12}$ let $F$ and $G$ be the unique functions in
$H_{1}\otimes H_{2}$ which restrict to $f$ and $g$ and which have minimum norm
subject to that condition. We then set
\[
\left\langle f,g\right\rangle _{H_{12}}=\left\langle F,G\right\rangle
_{H_{1}\otimes H_{2}}.
\]
To say the same thing in different words we map $H_{12}$ into $H_{1}\otimes
H_{2}$ by mapping each $h$ to the unique element of $H_{1}\otimes H_{2}$ which
restricts to $h$ and which is orthogonal to all functions which vanish on $D.$
The Hilbert space structure on $H_{12}$ is defined by declaring that map to be
a norm isometry$.$ It is a classical fact about RKHSs that $K_{12},$ the
kernel function for $H_{12},$ is given by
\[
K_{12}\left(  x,y\right)  =K_{1}(x,y)K_{2}(x,y).
\]
This leads to a relatively simple relation between the distance functions
$\delta_{1},$ $\delta_{2}$ and $\delta_{12}$ which we now compute. Pick
$x,y\in X.$ We have
\begin{align*}
1-\delta_{12}^{2}\left(  x,y\right)   &  =\frac{\left\vert K_{12}\left(
x,y\right)  \right\vert ^{2}}{K_{12}(x,x)K_{12}(y,y)}\\
&  =\frac{\left\vert K_{1}\left(  x,y\right)  \right\vert ^{2}}{K_{1}%
(x,x)K_{1}(y,y)}\frac{\left\vert K_{2}\left(  x,y\right)  \right\vert ^{2}%
}{K_{2}(x,x)K_{2}(y,y)}\\
&  =\left(  1-\delta_{1}^{2}\left(  x,y\right)  \right)  \left(  1-\delta
_{2}^{2}\left(  x,y\right)  \right)  .
\end{align*}
Hence
\[
\delta_{12}=\sqrt{\delta_{1}^{2}+\delta_{2}^{2}-\delta_{1}^{2}\delta_{2}^{2}%
}.
\]
\ This implies the less precise, but more transparent, estimates
\begin{equation}
\max\left\{  \delta_{1},\delta_{2}\right\}  \leq\delta_{12}\leq\delta
_{1}+\delta_{2}; \label{inequality}%
\end{equation}
with equality only in degenerate cases. Similar results hold for $\delta
_{1,2}$, the distance associated with $H_{1}\otimes H_{2}.$

A particular case of the previous equation is $H_{1}$= $H_{2}.$In that case
\[
\delta_{12}=\sqrt{2\delta_{1}^{2}-\delta_{1}^{4}}\geq\delta_{1}.
\]
This monotonicity, which for instance relates the $\delta$ associated with the
Hardy space with that of the Bergman space, is a special case of the more
general fact. If we have, for a set of positive $\alpha,$ a family of spaces
$H_{\alpha}$ and if there is a fixed function $K$ so that the kernel function
$K_{\alpha}$ for $H_{\alpha}$ is $K^{\alpha}$ then there is automatically a
monotonicity for the distance functions; if $\alpha^{\prime}>\alpha$ then
$\delta_{\alpha^{\prime}}\geq\delta_{\alpha}.$ This applies, for instance, to
the families $\left\{  \mathcal{H}_{\alpha}\right\}  $ and $\left\{  F_{\beta
}\right\}  $ introduced earlier. We also note in passing that in those two
families of examples there is also a monotonicity of the spaces; if
$\alpha<\alpha^{\prime}$ then there is a continuous, in fact a compact,
inclusion of $\mathcal{H}_{\alpha}$ into $\mathcal{H}_{\alpha^{\prime}};$
similarly for the $F_{\beta}.$

\section{The Skwarcy\'{n}ski Metric}

In \cite{K} Kobayshi studies the differential geometry of bounded domains,
$\Omega,$ in $\mathbb{C}^{n}.$ He begins with the observation that there was a
natural map of $\Omega$ into the projective space over the Bergman space of
$\Omega$. He then notes that either of two naturally occurring metrics on that
projective space could then be pulled back to $\Omega$ where they would be
useful tools. However looking back at his paper there was no particular use
made of the Bergman space beyond the fact that it was a RKHS. We will now
describe his constructions and see that they give expressions which are not
the same as $\delta$ but are closely related.

Suppose $H$ is a RKHS of functions on $X.$ Canonically associated with any
point $x\in X$ is a one dimensional subspace of $H,$ the span of the kernel
function $k_{x}$, or, what is the same thing, the orthogonal complement of the
space of functions in $H$ which vanish at $x.$ The projective space over $H,$
$P(H),$ is the space of one dimensional subspaces of $H.$ Hence for each $x\in
X$ we can use the span of $k_{x},$ $\left[  k_{x}\right]  $ to associate to
$x$ a point $p_{x}=\left[  k_{x}\right]  \in P(H).$ To understand the geometry
of this mapping we break in into two steps. First, we associate to each $x\in
X$ the set of vectors in the unit sphere, $S(H)$ that are in the span of
$k_{x}$; all these vectors of the form $e^{i\theta}\hat{k}_{x}$ for real
$\theta.$ The next step is to collapse this circle sitting in the unit
spheres, $\left\{  e^{i\theta}\hat{k}_{x}:\theta\in\mathbb{R}\right\}  ,$ to
the single point $p_{x}=\left[  k_{x}\right]  .$ In fact every point $p\in
P(H)$ is associated in this way to a circle $C(p)\subset$ $S(H)$ and distinct
points correspond to disjoint circles. We now use the fact the distance
function of $H$ makes $S(H)$ a metric space and use that metric to put the
quotient metric on $P(H).$ That is, define a metric $\hat{\delta}$ on $P(H),$
sometimes called the Cayley metric, by:%

\[
\hat{\delta}(p,q)=\inf\left\{  \left\Vert r-s\right\Vert :r\in C(p),s\in
C(q)\right\}
\]

On the subset $\left\{  p_{x}:x\in X\right\}  $ we have explicit descriptions
of the circles $C(p_{x})$ and we compute
\begin{align*}
\hat{\delta}(p_{x},p_{y})  &  =\inf\left\{  \left\Vert e^{i\theta}\hat{k}%
_{x}-e^{i\eta}\hat{k}_{x}\right\Vert :\theta,\eta\in\mathbb{R}\right\} \\
&  =\inf\sqrt{2-2\operatorname{Re}e^{i(\theta-\eta)}\left\langle \hat{k}%
_{x},k_{y}\right\rangle }\\
&  =\sqrt{2}\sqrt{1-\left\vert \left\langle \hat{k}_{x},\hat{k}_{y}%
\right\rangle \right\vert }.
\end{align*}
We now pull this metric back to $X$ and, with slight abuse of notation,
continue to call it $\hat{\delta};$%
\begin{equation}
\hat{\delta}(x,y)=\sqrt{2}\sqrt{1-\left\vert \left\langle \hat{k}_{x},\hat
{k}_{y}\right\rangle \right\vert } \label{def2}%
\end{equation}
Thus the map of $X$ into $P(H)$ which sends $x$ to $p_{x}$ is a isometry $X$
with the metric $\hat{\delta}$ into $P(H)$ with its metric as a quotient of
$S(H)$.

This metric was studied further by Skwarcy\'{n}ski in \cite{S} and then in
collaboration with Mazur and Pflug in \cite{MPS}. In \cite{JP} it is referred
to as the Skwarcy\'{n}ski metric.

The second metric Kobayshi introduces in this context is again obtained by
putting a natural metric on $P(H)$ and then, again, pulling it back to $X.$
The Fubini-Study metric is a natural Kahler metric in finite dimensional
projective space and Kobayshi extends that definition to the generally
infinite dimensional $P(H)$. We denote by $\check{\delta}$ the metric obtained
by restricting the Fubini-Study metric to the image of $X$ in $P(H)$ and then
pulling the metric back to $X.$

In the small these three metrics are almost the same. If one is small so are
the others and, in fact, setting $\delta(x,y)=\delta_{1},$ $\hat{\delta
}(x,y)=\delta_{2},$ $\check{\delta}(x,y)=\delta_{3}$ we have that
\begin{equation}
\max_{i=1,2,3}\left\{  \delta_{i}\right\}  =O(\varepsilon)\Longrightarrow
\max_{i,j=1,2,3}\left\{  |\delta_{i}-\delta_{j}|\right\}  =O(\varepsilon^{3}).
\label{same}%
\end{equation}
The comparison between $\delta_{1}$ and $\delta_{2}$ follows from
(\ref{def1}), (\ref{def2}) and Taylor's theorem. The comparison of $\delta
_{2}$ and $\delta_{3},$ which is not difficult, is given in \cite[pg. 282]{K}.
Comparing (\ref{def1}) and (\ref{def2}) also allows one to see that
\[
\lim_{\delta(x,y)\rightarrow1}\frac{\delta(x,y)}{\hat{\delta}(x,y)/\sqrt{2}%
}=1
\]

\begin{remark}
In mathematical physics, for certain choices of $H,$ the map from $X$ into the
projective space is related to coherent state quantization. In that context
some of the quantities we have been working with, or will be below, are given
physical names/interpretations. For instance $|<\hat{k}_{x},\hat{k}_{y}%
>|^{2}=1-\delta(x,y)^{2}$ is the probability density for transition from the
quantum state $\left[  k_{x}\right]  $ to the state $\left[  k_{y}\right]  .$
See, for instance, \cite{O1}, \cite{O2}, \cite{AE}, and \cite{PV}.
\end{remark}

\section{Differential Geometric Metrics}

In this section we describe the relationship between the distance functions we
introduced $\delta,$ the associated length functions and inner metrics, and a
Riemannian metric built using the kernel functions. Throughout this section we
suppose that $H$ is a RKHS of holomorphic functions on a domain $X $ in
$\mathbb{C}.$ We further suppose that $H$ is nondegenerate in the sense that
$\forall x,y\in X,$ $\exists h,k\in H$ with $h(x)\neq0,k(x)\neq k(x).$ These
restrictions are much more than is needed for most of what follows but some
restrictions are necessary. For instance the results in the next subsection
require that the kernel function $K(x,y)$ be sufficiently smooth so that one
can apply to it the second order Taylor's theorem; some of the results in the
third subsection are specific to one complex variable.

\subsection{Results of Mazur, Pflug, and Skwarcy\'{n}ski}

In an earlier section we described how, for distance functions on the unit
disk, one could pass from a distance to the associated inner distance. That
discussion was not specific to the disk and we now apply those constructions
to distances defined on $X.$ That is, given a distance function $\mathcal{D}$
we define the length of a curve $\gamma$ by $\ell_{\mathcal{D}}(\gamma
)=\sup\left\{  \sum\mathcal{D}(\gamma(t_{i})\gamma(t_{i+1}))\right\}  $ and
the inner distance induced by $\mathcal{D}$ is given by $\mathcal{D}^{\ast
}(x,y)=\inf\left\{  \ell_{\mathcal{D}}(\gamma):\gamma\text{ is a curve from
}x\text{ to }y\right\}  .$ Clearly $\mathcal{D}^{\ast}\geq\mathcal{D}$ and if
the two functions are equal we say $\mathcal{D}$ is an inner distance. For
example Euclidean distance on the plane is an inner distance; the
pseudohyperbolic distance in the disk is not an inner distance, its induced
inner distance is the hyperbolic distance.

The reproducing kernel for the Bergman space of the disk, $\mathcal{H}%
_{2}=B(\mathbb{D)},$ the \textit{Bergman kernel} is $K(x,y)=\left(  1-x\bar
{y}\right)  ^{-2}.$ Using it we can construct a Riemannian metric on the disk
through
\begin{align*}
ds^{2}  &  =\frac{\partial^{2}}{\partial z\partial\bar{z}}\log k_{z}%
(z)\left\vert dz\right\vert ^{2}=\frac{1}{4}\Delta\log k_{z}(z)\left\vert
dz\right\vert ^{2}\\
&  =\frac{2}{\left(  1-\left\vert z\right\vert ^{2}\right)  ^{2}}\left\vert
dz\right\vert ^{2}%
\end{align*}
which is a constant multiple of the density we saw earlier when introducing
the hyperbolic metric on the disk.

More generally, if $X$ is a bounded domain in $\mathbb{C}$ and $H$ is the
Bergman space of $X,$ $H=B(X),$ and $\left\{  k_{z}\right\}  $ are the
reproducing kernels then the formula $ds^{2}=\partial\bar{\partial}\log
k_{z}(z)\left\vert dz\right\vert ^{2}$ defines a Riemannian metric on $X,$ the
so called Bergman metric. There is also an extension of this construction to
domains in $\mathbb{C}^{n}.$

In \cite{MPS} Mazur, Pflug, and Skwarcy\'{n}ski prove three theorems. Suppose
that $X$ is a bounded domain the $\mathbb{C}$ (they actually work with
$\mathbb{C}^{n}).$ Let $H$ be the Bergman space of $X$. For a curve $\gamma$
in $X$ let $\ell_{B}(\gamma)$ be its length measured using the Bergman metric.
For $x,y\in X$ denote the Bergman distance between them by $\delta
_{B}(x,y)=\inf\left\{  \ell_{B}(\gamma):\gamma\text{ a smooth curve from
}x\text{ to }y\right\}  .$

\begin{proposition}[\cite{K}, \cite{MPS}]
\hfill

\begin{enumerate}
\item For any smooth curve $\gamma$
\[
\ell_{\delta}(\gamma)=\ell_{\hat{\delta}}(\gamma)=\ell_{\check{\delta}}%
(\gamma)=\frac{1}{2}\ell_{B}(\gamma).
\]

\item
\[
\delta^{\ast}=\hat{\delta}^{\ast}=\check{\delta}^{\ast}=\check{\delta}%
=\frac{1}{2}\delta_{B}.
\]

\item
\[
\forall x,y\in X,\text{ }x\neq y,\text{ }\hat{\delta}(x,y)<\hat{\delta}^{\ast
}\left(  x,y\right)  .
\]

\end{enumerate}
\end{proposition}

In particular (up to a constant factor) the Bergman distance is the inner
distance generated by our metrics and the metric $\hat{\delta}$ is never an
inner metric.

The results for $\hat{\delta}$ are proved in \cite{MPS} (with an unfortunate
typo in the statement of Theorem 2 there). We noted that locally the three
distances in (\ref{same}) agree to third order. Hence the three metrics
generate the same length function and same inner distance. Because of this the
results for $\delta$ and $\check{\delta}$ follow from the ones for
$\hat{\delta}.$ The equality $\check{\delta}^{\ast}=\check{\delta}$, i.e. the
statement that $\check{\delta}$ is an inner distance, follows from the
discussion in \cite{K}.

The discussion in \cite{K} \cite{MPS} is given for Bergman spaces,
$B(\Omega).$ However the results hold in more generality. Given $X$ and $H$ of
the type we are considering, with reproducing kernels $\left\{  k_{z}\right\}
,$ there is a standard associated Riemannian metric given by
\begin{equation}
ds_{H}^{2}=\left(  \frac{\partial^{2}}{\partial z\partial\bar{z}}\log
k_{z}(z)\right)  \left\vert dz\right\vert ^{2}. \label{ds}%
\end{equation}
If we define functions $k_{z}^{(1)}$ in $H$ by requiring that for all $f\in H
$ $\left\langle f,k_{z}^{(1)}\right\rangle =f^{\prime}(z)$ then one can
compute that
\[
ds_{H}^{2}=\frac{\left\Vert k_{z}^{(1)}\right\Vert ^{2}\left\Vert
k_{z}^{\text{ }}\right\Vert ^{2}-\left\vert \left\langle k_{z}^{(1)}%
,k_{z}\right\rangle \right\vert ^{2}}{\left\Vert k_{z}^{\text{ }}\right\Vert
^{4}}\left\vert dz\right\vert ^{2}.
\]
One can then define the Bergman style length of a curve $\gamma,$ $\ell
_{BS}(\gamma),$ to be the length of $\gamma$ measured using $ds_{H}$ and can
set
\[
\delta_{BS}\left(  x,y\right)  =\inf\left\{  \ell_{BS}(\gamma):\gamma\text{ a
smooth curve connecting }x\text{ to }y\right\}  .
\]

We have defined $\delta=\delta_{H}$ for such an $H.$ We define $\hat{\delta
}=\hat{\delta}_{H}$ using (\ref{def2}). We define $\check{\delta}%
=\check{\delta}_{H}$ by following Kobayashi's prescription. We have a map of
$X$ into the $P(H)$ which sends $x$ to $p_{x}=\left[  k_{x}\right]  .$ We use
that map to we pull back the Fubini-Study metric on $P(H)$ back to $X$ and
call the resulting metric $\check{\delta}.$

\begin{proposition}
\label{inner}
\hfill
\begin{enumerate}
\item For any smooth curve $\gamma$
\[
\ell_{\delta}(\gamma)=\ell_{\hat{\delta}}(\gamma)=\ell_{\check{\delta}}%
(\gamma)=\frac{1}{2}\ell_{BS}(\gamma).
\]

\item
\[
\delta^{\ast}=\hat{\delta}^{\ast}=\check{\delta}^{\ast}=\check{\delta}%
=\frac{1}{2}\delta_{BS}.
\]

\item
\[
\forall x,y\in X,x\neq y,\hat{\delta}(x,y)<\hat{\delta}^{\ast}\left(
x,y\right)  .
\]

\item
\[
\forall x,y\in X,x\neq y,\delta(x,y)<\delta^{\ast}\left(  x,y\right)  .
\]

\end{enumerate}

\begin{proof}
The proof in \cite{MPS} of versions of the first two statements are based on
the second order Taylor approximations to the kernel functions; hence those
proofs apply here as does the discussion in \cite{K} which shows that
$\check{\delta}$ is an inner metric. The third statement follows from the
proof in \cite{MPS} together with the fact that for any $a,b,c\in X$ we have
the strict inequality $\hat{\delta}(a,c)<\hat{\delta}(a,b)+\hat{\delta}(b,c).$
The proof of that in \cite{MPS} does not use the fact that $H$ is a Bergman
space, rather it uses the fact that $\hat{\delta}$ was obtained by pulling the
Cayley metric back from projective space, which also holds in our context.{}
The fourth statement also follows from the proof in \cite{MPS} if we can
establish the fact that for any $a,b,c\in X$ we have the strict inequality
$\delta(a,c)<\delta(a,b)+\delta(b,c).$ We will obtain that from Proposition
\ref{norm}. We need to rule out the possibility that
\begin{equation}
\left\Vert P_{a}-P_{c}\right\Vert =\left\Vert P_{a}-P_{b}\right\Vert
+\left\Vert P_{b}-P_{c}\right\Vert . \label{equality}%
\end{equation}
The operator $P_{a}-P_{c}$ is a rank two self adjoint operator. Hence hence it
has a unit eigenvector, $v,$ with%
\[
\left\Vert \left(  P_{a}-P_{c}\right)  (v)\right\Vert =\left\Vert P_{a}%
-P_{c}\right\Vert .
\]
For both of the two previous equalities to hold we must also have
\[
\left\Vert \left(  P_{a}-P_{b}\right)  (v)\right\Vert =\left\Vert P_{a}%
-P_{b}\right\Vert \text{, }\left\Vert \left(  P_{b}-P_{c}\right)
(v)\right\Vert =\left\Vert P_{b}-P_{c}\right\Vert
\]
and hence $v$ must also be an eigenvector of $P_{a}-P_{b}$ and also of
$P_{b}-P_{c}.$ 

However in our analysis in Proposition \ref{norm} we
saw that an eigenvector for an operator $P_{x}-P_{y}$ must be in $%
{\textstyle\bigvee}
\left\{  k_{x},k_{y}\right\}  ,$ the span of $k_{x}$ and $k_{y}.$ Thus
\[
v\in%
{\textstyle\bigvee}
\left\{  k_{a},k_{c}\right\}
{\textstyle\bigcap}
{\textstyle\bigvee}
\left\{  k_{a},k_{b}\right\}
{\textstyle\bigcap}
{\textstyle\bigvee}
\left\{  k_{b},k_{c}\right\}  =\left\{  0\right\}  ,
\]
a contradiction.
\end{proof}
\end{proposition}

\subsection{The Berezin Transform and Lipschitz Estimates}

Suppose $A$ is a bounded linear map of $H$ to itself. The Berezin transform of
$A$ is the scalar function defined on $X$ by the formula%
\[
\hat{A}(x)=\left\langle A\hat{k}_{x},\hat{k}_{x}\right\rangle .
\]
For example, if $P_{a}$ is the orthogonal projection onto the span of $k_{a}$
then%
\[
\widehat{P_{a}}(x)=1-\delta^{2}\left(  a,x\right)  .
\]
Also, recalling Proposition \ref{multiplier}, we have the following. Suppose
$m,n\in M(H)$ and that $M$ and $N$ are the associated multiplication operators
on $H.$ We then have
\[
\widehat{MN^{\ast}}(x)=\widehat{M}(x)\overline{\widehat{N}(x)}=m(x)\overline
{n(x)}.
\]

Coburn showed that the metric $\delta$ is a natural tool for studying the
smoothness of Berezin transforms.

\begin{proposition}
[Coburn \cite{CO2}]If $A$ is a bounded linear operator on $H,$ $x,y\in X$ then%
\begin{equation}
\left\vert \hat{A}(x)-\hat{A}(y)\right\vert \leq2\left\Vert A\right\Vert
\delta(x,y). \label{berezin}%
\end{equation}
Thus, also, if $m\in M(H)$ and $M$ is the associated multiplication operator
then
\begin{equation}
\left\vert m(x)-m(y)\right\vert \leq2\left\Vert M\right\Vert \delta(x,y).
\label{mult}%
\end{equation}
Estimate (\ref{berezin}) is sharp in the sense that given $H,$ $x,$ and $y$
one can select $A$ so that equality holds.

\begin{proof}
It is standard that if $A$ is bounded and $T$ is trace class then $AT$ is
trace class and $\left\vert \operatorname*{Trace}(AT)\right\vert
\leq\left\Vert A\right\Vert \left\Vert T\right\Vert _{\operatorname*{Trace}}.$
Recall that $P_{x}$ is the orthogonal projection onto the span of $k_{x}.$
Direct computation shows $\operatorname*{Trace}(AP_{x})=\hat{A}(x).$ Thus
$\left\vert \hat{A}(x)-\hat{A}(y)\right\vert \leq2\left\Vert A\right\Vert
\left\Vert P_{x}-P_{y}\right\Vert _{\operatorname*{Trace}}.$ The proofs of
(\ref{berezin}) and (\ref{mult}) are then completed by taking note of
Proposition \ref{norm}. To see that the result is sharp evaluate both sides
for the choice $A=P_{x}-P_{y}.$ Details of those computation are in the proof
of Proposition \ref{norm}.
\end{proof}
\end{proposition}

\begin{remark}
By analysis of the two by two Pick matrix of $M^{\ast}$ one sees that
(\ref{mult}) is not sharp.
\end{remark}

Suppose $\gamma:\left(  0,1\right)  \rightarrow X$ is a continuous curve in
$X$ and that $f$ is a function defined in a neighborhood of the curve. We
define the variation of $f$ along the curve to be%
\[
\operatorname*{Var}_{\gamma}(f)=\sup\left\{  \sum_{i=1}^{n-1}\left\vert
f(t_{i})-f(t_{i+1})\right\vert :0<t_{1}<\cdot\cdot\cdot<t_{n}<1,\text{
}n=1,2,...\right\}  .
\]

\begin{corollary}
With $H,X,\gamma,A$ as above:
\[
\operatorname*{Var}_{\gamma}(\hat{A})\leq2\left\Vert A\right\Vert \ell
_{\delta}(\gamma)=\left\Vert A\right\Vert \ell_{BS}(\gamma).
\]

\begin{proof}
If we start with a sum estimating the variation of $\hat{A}$ and apply the
previous proposition to each summand we obtain the first estimate. The second
inequality follows from the first and Proposition \ref{inner}.
\end{proof}
\end{corollary}

These issues are also studied when $X$ has dimension greater than 1 and there
is a rich relationship between the properties of Berezin transforms and the
differential geometry associated with $ds_{BS}^{2},$ \cite{CO}, \cite{CO2},
\cite{CL}, \cite{EZ}, \cite{EO}, \cite{BO}.

\subsection{Limits Along Curves\label{lac}}

For the most commonly considered examples of a RKHS any curve in $X$ which
leaves every compact subset of $X$ has infinite length when measured by any of
the length functions we have been considering. However this not always the
case. For example, if $X$ is the open unit disk and $H$ is defined by the
kernel function
\[
K(z,y)=\frac{2-z-\bar{y}}{1-\bar{y}z}%
\]
then straightforward estimates show that along the positive axis%
\[
ds_{BS}=\left(  \frac{1}{2\sqrt{1-r}}+o\left(  \frac{1}{\sqrt{1-r}}\right)
\right)  dr;
\]
hence the curve $[0,1)$ has finite length. For more discussion of this see
\cite{Mc}. This suggests there may be interesting limiting behavior as one
traverses the curve. Suppose $f$ is a function defined on $\gamma$ except at
the endpoints. Straightforward analysis then shows that if
$\operatorname*{Var}_{\gamma}(f)<\infty$ then $f$ has limiting values along
$\gamma$ as one approaches the endpoints. Thus

\begin{corollary}
Given $H,X;$ suppose $\gamma:[0,1)\rightarrow X$ and $\ell_{BS}(\gamma
)<\infty$ then if $m$ is any multiplier of $H$ or, more generally, if $\hat
{A}$ is the Berezin transform of any bounded operator on $H,$ then these
functions have limits along $\gamma;$%
\[
\exists\lim_{t\rightarrow1}m(\gamma(t));\text{ }\exists\lim_{t\rightarrow
1}\hat{A}(\gamma(t)).
\]
Furthermore there are choices of $m$ or $A$ for which the limits are not zero.
\end{corollary}

This invites speculation that something similar might be true for $H$, however
that situation is more complicated. If we rescale the space then we do not
change $\delta$ and thus don't change the class of curves of finite length.
However rescaling certainly can change the validity of statements that
functions in the space have limits along certain curves. Thus the possibility
of rescaling a space is an obstacle to having a result such as the previous
corollary for functions in $H.$ McCarthy showed in \cite{Mc} that, in some
circumstances, this is the only obstacle. We state a version of his result but
will not include the proof.

\begin{proposition}
[McCarthy \cite{Mc}]Suppose $H$ is a RKHS of holomorphic functions on the
disk and that $\gamma:[0,1)\rightarrow\mathbb{D}$ and $\ell_{BS}%
(\gamma)<\infty.$ There is a holomorphic function $G$ such that every function
in the rescaled space $GH $ has a limit along $\gamma.$ That is
\[
\exists G,\text{ }\forall h\in H,\text{ }\exists\lim_{t\rightarrow1}%
G(\gamma(t))h(\gamma(t)).
\]
Furthermore $G$ can be chosen so that for some $h\in H$ this limit is not zero.
\end{proposition}

McCarthy's work was part of an investigation of an interface between operator
theory and differential geometry that goes back (at least) to the work of
Cowen and Douglas \cite{CD}. They showed that one could associate to certain
operators a domain in the plane and a Hermitian holomorphic line bundle in
such a way that the domain and curvature of the line bundle formed a complete
unitary invariant for the operator. That is, two such operators are unitarily
equivalent if and only if the domains in the plane agree and the curvatures on
the line bundles agree. An alternative presentation of their approach yields
RKHSs of holomorphic functions which satisfy certain additional conditions. In
this viewpoint the statement about unitary equivalence becomes the statement
that for certain pairs of RKHSs of holomorphic functions, if their metrics
$\delta$ are the same then each space is a holomorphic rescaling of the other.
The connection between the two viewpoints is that in terms of the kernel
function of the RKHS, the curvature function at issue is
\[
\mathcal{K}(z)=-\frac{\partial^{2}}{\partial z\partial\bar{z}}\log k_{z}(z).
\]
Thus $\mathcal{K}$ or, equivalently $\delta,$ contains a large amount of
operator theoretic information. However that information is not easy to
access; which is why McCarthy's result is so nice and why these relations seem
worth more study.

\subsection{RKHS's With Complete Nevanlinna Pick Kernels}

There is a class of RKHS's which are said to have complete Nevanlinna Pick
kernels or complete NP kernels. The classical Hardy space is the simplest. The
class is easy to define, we will do that in a moment, but the definition is
not very informative. A great deal of work has been done in recent years
studying this class of spaces. The book \cite{AM} by Agler and McCarthy is a
good source of information. In this and the next section we will see that for
this special class of RKHS's the function $\delta$ has additional properties.

We will say that the RKHS $H$ has a complete NP kernel if there are functions
$\left\{  b_{i}\right\}  _{i=1}^{\infty}$ defined on $X$ so that
\[
1-\frac{1}{K(x,y)}=\sum_{i=1}^{\infty}b_{i}(x)\overline{b_{i}(y)};
\]
that is, if the function $1-1/K$ is positive semidefinite.

Of the spaces in our earlier list of examples the Hardy spaces, generalized
Dirichlet spaces, and the Dirichlet space have complete NP kernels. This is
clear for the Hardy space, for the generalized Dirichlet spaces if follows
from using the Taylor series for $1/K$, and for the Dirichlet space there is
some subtlety involved in the verification. On the other hand neither the
generalized Bergman spaces nor the Fock spaces have a complete NP kernel.

Suppose $H$ is a RKHS of functions on $X$ and $x,y\in X,$ $x\neq y.$ Let
$G=G_{x,y}$ be the multiplier of $H$ of norm $1$ which has $G(x)=0$ and
subject to those conditions maximizes $\operatorname{Re}G(y).$

\begin{proposition}
\label{multiplier modulus of continuity}
\[
\operatorname{Re}G_{x,y}\left(  y\right)  \leq\delta(x,y).
\]
If $H$ has a complete NP kernel then%
\[
\operatorname{Re}G_{x,y}\left(  y\right)  =\delta(x,y).
\]
and $G$ is given uniquely by%
\[
G_{x,y}\left(  \cdot\right)  =\delta_{H}(x,y)^{-1}\left(  1-\frac
{k_{y}(x)k_{x}(\cdot)}{k_{x}(x)k_{y}(\cdot)}\right)  .
\]

\begin{proof}
\cite[(5.9) pg 93]{Sa}
\end{proof}
\end{proposition}

\begin{remark}
The multipliers of $H$ form a commutative Banach algebra $\mathcal{M}.$ A
classical metric on the spectrum of such an algebra, the Gleason metric, is
given by
\[
\delta_{G}(\alpha,\beta)=\left\{  \sup\operatorname{Re}\alpha(M):M\in
\mathcal{M},\left\Vert M\right\Vert =1,\beta(M)=0\right\}  ,
\]
\cite{G}, \cite{L}, \cite{BW}. The points of $X$ give rise to elements of the
spectrum via $\hat{x}(G)=G(x).$ Thus if $H$ has a complete NP kernel then
$\delta$ agrees with the Gleason metric:
\[
\forall x,y\in X,\text{ }\delta_{G}(\hat{x},\hat{y})=\delta(x,y).
\]

\end{remark}

\subsubsection{Generalized Blaschke Products}

If $H=\mathcal{H}_{1},$ the Hardy space, when we compute $G$ we get%

\begin{align*}
G_{x,y}\left(  z\right)   &  =\left\vert \frac{1-\bar{x}y}{y-x}\right\vert
\left(  1-\frac{(1-\left\vert x\right\vert ^{2})\left(  1-\bar{y}z\right)
}{\left(  1-\bar{y}x\right)  \left(  1-\bar{x}z\right)  }\right) \\
&  =\left\vert \frac{1-\bar{x}y}{y-x}\right\vert \frac{\bar{y}-\bar{x}}%
{1-\bar{y}x}\frac{z-x}{1-\bar{x}z}\\
&  =e^{i\theta}\frac{z-x}{1-\bar{x}z}.
\end{align*}
Thus $G_{x,y}$ is a single Blaschke factor which vanishes at $x$ and is
normalized to be positive at the base point $y.$

Suppose now we have a RKHS $H$ which has a complete NP kernel and let us
suppose for convenience that it is a space of holomorphic functions on
$\mathbb{D}$. Suppose we are given a set $S=\left\{  x_{i}\right\}
_{i=1}^{\infty}\subset\mathbb{D}$ and we want to find a function in $H$ and/or
$M(H)$ whose zero set is exactly $S.$ We could use the functions $G$ just
described and imitate the construction of a general Blaschke product from the
individual Blaschke factors. That is, pick $x_{0}\in\mathbb{D}\smallsetminus
S$ and consider the product%
\begin{equation}
B(\zeta)=B_{S,x_{0}}\left(  \zeta\right)  =\prod_{i=1}^{\infty}G_{x_{i},x_{0}%
}\left(  \zeta\right)  . \label{product}%
\end{equation}
If the product converges then $B\left(  \zeta\right)  $ will be a multiplier
of norm at most one and its zero set will be exactly $S.$

The multiplier norm dominates the supremum so the factors in (\ref{product})
have modulus less than one. Hence the product either converges to a
holomorphic function with zeros only at the points of $S$ or the product
diverges to the function which is identically zero. The same applies to the
function $B^{2}(\zeta).$ We test the convergence of that product by evaluation
at $x_{0}$ and, recalling that $G_{z,y}(y)=\delta_{H}\left(  z,y\right)  ,$ we see%

\begin{align}
B_{S,x_{0}}^{2}\left(  x_{0}\right)   &  =\prod_{i=1}^{\infty}\delta^{2}%
(x_{i},z_{0}),\label{prod1}\\
&  =\prod_{i=1}^{\infty}\left(  1-\frac{\left\vert k_{x_{0}}(x_{i})\right\vert
^{2}}{\left\Vert k_{x_{0}}\right\Vert ^{2}\left\Vert k_{x_{i}}\right\Vert
^{2}}\right)  . \label{prod3}%
\end{align}

Recalling the conditions for absolute convergence of infinite products, we
have established the following:

\begin{proposition}
The generalized Blaschke product $B_{S,x_{0}}$ converges to an element of
$M(H)$ of norm at most one and with zero set exactly $S$ if and only if the
following two equivalent conditions hold
\begin{align}
\prod_{i=1}^{\infty}\delta^{2}(x_{i},x_{0})  &  >0\nonumber\\
\sum_{i=1}^{\infty}\frac{\left\vert k_{x_{0}}(x_{i})\right\vert ^{2}%
}{\left\Vert k_{x_{0}}\right\Vert ^{2}\left\Vert k_{x_{i}}\right\Vert ^{2}}
&  <\infty. \label{sum}%
\end{align}
If the conditions do not hold then $B_{S,x_{0}}$ is the zero function.
\end{proposition}

\begin{description}
\item[Remark] If $1\in H$ then $M(H)\subset H$ and the proposition gives
conditions that insure that there is a function in $H$ with zero set exactly
$S.$
\end{description}

\begin{corollary}
A sufficient condition for the set $S$ to be a zero set for the Hardy space,
$H_{1},$ or of $M(H_{1})$ which is known to be the space of bounded analytic
functions in the disk, is that
\begin{equation}
\sum\left(  1-\left\vert x_{i}\right\vert ^{2}\right)  <\infty.
\label{Blaschke}%
\end{equation}
A sufficient condition for the set $S$ to be a zero set for the Dirichlet
space $H_{0}$ or of $M(H_{0})$ is that%
\begin{equation}
\sum\log\left(  \frac{1}{1-\left\vert x_{i}\right\vert ^{2}}\right)  <\infty.
\label{Shapiro and Shields}%
\end{equation}

\begin{proof}
These are just the conclusions of the previous proposition applied to the
Hardy space and the Dirichlet space with the choice of the origin for the basepoint.
\end{proof}
\end{corollary}

\begin{remark}
Condition (\ref{Blaschke}) is the Blaschke condition which is well known to be
necessary and sufficient for $S$ to be the zero set of a function in
$\mathcal{H}_{1}$ or $H^{\infty}.$

The condition for $S$ to be a zero set for the Dirichlet space was first given
by Shapiro and Shields \cite{SS} and the argument we gave descends from
theirs. It is known that this condition is necessary and sufficient if
$S\subset\left(  0,1\right)  $ but is not necessary in general.
\end{remark}

\begin{remark}
In the next subsection we note that any $H$ with a complete NP kernel is
related to a special space of functions in a complex ball. Using that
relationship one checks easily that the convergence criteria in the
Proposition is a property of the set $S$ and is independent of the choice of
$x_{0}.$
\end{remark}

\subsubsection{The Drury Arveson Hardy Space and Universal Realization}

For $n=1,2,...$ we let $\mathbb{B}^{n}$ denote the open unit ball in
$\mathbb{C}^{n}.$ We allow $n=\infty$ and interpret $\mathbb{B}^{\infty}$ to
be the open unit ball of the one sided sequence space $\ell^{2}\left(
\mathbb{Z}_{+}\right)  .$ For each $n$ we define the $n-$dimensional Drury
Arveson Hardy space, $D_{n}$ to be the RKHS on $\mathbb{B}^{n}$ with kernel
function%
\[
K_{n}(x,y)=\frac{1}{1-\left\langle x,y\right\rangle }.
\]
Thus when $n=1$ we have the classical Hardy space.

For each $n$ the kernel function $K$ is a complete NP kernel. The spaces
$D_{n}$ are universal in the sense that any other RKHS with a complete NP
kernel can be realized as a subspace of some $D_{n}.$ If $H$ is a RKHS on $X$
and $H$ has a complete NP kernel then there is for some $n,$ possibly
infinite, a mapping $\gamma:X\rightarrow$ $\mathbb{B}^{n}$ and a nonvanishing
function $b$ defined on $X$ so that $K_{H},$ the kernel function for the space
$H,$ is given by
\begin{equation}
K_{H}(x,y)=b(x)\overline{b(y)}K_{n}(\gamma(x),\gamma(y))=\frac{b(x)\overline
{b(y)}}{1-\left\langle \gamma(x),\gamma(y)\right\rangle }. \label{universal}%
\end{equation}
There is no claim of smoothness for $\gamma.$ All this is presented in
\cite{AM}.

The map $\gamma$ can be used to pull back the pseudohyperbolic metric from
$\mathbb{B}^{n}$ to produce a metric on $X.$ First we recall the basic facts
about the pseudohyperbolic metric $\mathbb{B}^{n}.$ Details about the
construction of the metric and its properties can be found in \cite{DW}; the
discussion there is for finite $n$ but the rudimentary pieces of theory we
need for infinite $n$ follow easily from the same considerations.

The pseudohyperbolic metric $\rho$ on the unit disk, $\mathbb{B}^{1},$ can be
described as follows. The disk possesses a transitive group of biholomorphic
automorphisms, $G=\left\{  \phi_{\alpha}\right\}  _{\alpha\in A}. $ Given a
pair of points $z,w\in\mathbb{B}^{1}$ select a $\phi_{\alpha}\in G$ so that
$\phi_{a}(z)=0.$ The quantity $\left\vert \phi_{\alpha}(w)\right\vert $ can be
shown to be independent of the choice of $\phi_{\alpha}$ and we define
$\rho(z,w)=\left\vert \phi_{\alpha}(w)\right\vert .$

In this form the construction generalizes to $\mathbb{B}^{n}.$ The $n-$ball
has a transitive group of biholomorphic automorphisms, $\mathcal{G}=\left\{
\theta_{\beta}\right\}  _{\beta\in B}.$ Given a pair of points $z,w\in
\mathbb{B}^{n}$ select a $\theta_{\beta}\in\mathcal{G}$ so that $\theta
_{\beta}(z)=0.$ The quantity $\left\vert \theta_{\beta}(w)\right\vert $ can be
shown to be independent of the choice of $\theta_{\beta}$ and we define
$\rho_{n}(z,w)=\left\vert \theta_{\beta}(w)\right\vert .$ The only difference,
and that is hidden by our notation, is that now $|$ $\cdot$
$\vert$
denotes the Euclidean length of a vector rather than the modulus of a scalar.
The function $\rho_{n}$ can be shown to be a metric and to have the expected
properties including invariance under $\mathcal{G}$ and having an induced
inner metric $\rho_{n}^{\ast}$ that, up to a scalar factor, agrees with the
distance induced by the Poincare-Bergman metric tensor. Particularly important
for our purposes is that there is an analog of (\ref{magic}) \cite[pg 67]{DW}.
For $z,w\in\mathbb{B}^{n}$
\begin{equation}
1-\frac{(1-\left\vert z\right\vert ^{2})(1-|w|^{2})}{\left\vert 1-\left\langle
z,w\right\rangle \right\vert ^{2}}=\rho_{n}^{2}\left(  z,w\right)  .
\label{magic2}%
\end{equation}
An immediate consequence of the definition of $\delta,$ the relationship
(\ref{universal}), and the identity (\ref{magic2}) is that the metric $\delta$
on $X$ is the pull back of $\rho_{n}$ by $\gamma.$ Put differently $\gamma$ is
an isometric map of $\left(  X,\delta_{H}\right)  $ into $\left(
\mathbb{B}^{n},\rho_{\mathbb{B}^{n}}\right)  .$ In particular the $\delta$
metric on the Drury-Arveson space is the pseudohyperbolic metric on
$\mathbb{B}^{n}:$ $\delta_{D_{n}}=\rho_{\mathbb{B}^{n}}.$

\section{Invariant Subspaces and Their Complements}

Suppose we are given RKHSs on a set $X$ and linear maps between them. We would
like to use the $\delta$s on$\ X$ to study the relation between the function
spaces and to study the linear maps. The goal is broad and vague. Here we just
report on a few very special cases.

We will consider a RKHS $H$ of functions on a set $X,$ a closed multiplier
invariant subspace $J$ of $H;$ that is we require that if $j\in J$ and $m$ is
a multiplier of $H$ then $mj\in J.$ We will also consider the $J^{\perp},$ the
orthogonal complement of $J.$ The spaces $J$ and $J^{\perp}$ are RKHSs on $X$
and we will be interested the the relationship between the metrics $\delta
_{H},\delta_{J},$ and $\delta_{J^{\perp}}.$ Because we are working with a
subspace and its orthogonal complement there is a simple relation between the
kernel functions. Let $\left\{  k_{x}\right\}  $ be the kernel function $H,$
$\left\{  j_{x}\right\}  $ those for $J$ and $\left\{  j_{x}^{\perp}\right\}
$ those for $J^{\perp}.$ We then have, $\forall x\in X$%
\begin{equation}
k_{x}=j_{x}+j_{x}^{\perp}. \label{add}%
\end{equation}
In terms of $P,$ the orthogonal projection of $H$ onto $J,$ and $P^{\perp
}=I-P,$ we have $j_{x}=Pk_{x},$ $j_{x}^{\perp}=P^{\perp}k_{x}.$

\subsection{The Hardy Space}

We begin with $\mathcal{H}_{1},$ the Hardy space. In that case there is a good
description of the invariant subspaces, the computations go smoothly, and the
resulting formulas are simple. If $J$ is an invariant subspace then there is
an inner function $\Theta_{J}$ so that $J=\Theta_{J}\mathcal{H}_{1}. $ The
kernel functions are, for $z,w\in\mathbb{D}$, $\Theta_{J}(z)\neq0 $ are given
by
\[
j_{z}(w)=\frac{\overline{\Theta_{J}(z)}\Theta_{J}(w)}{1-\bar{z}w}.
\]
Thus if $\Theta_{J}(z)\Theta_{J}(z^{\prime})\neq0$ then
\[
\delta_{J}\left(  z,z^{\prime}\right)  =\delta_{\mathcal{H}_{1}}\left(
z,z^{\prime}\right)  .
\]
In the other cases, by our convention, $\delta_{J}$ is undefined.

Taking into account the formula for $j_{z}$ and (\ref{add}) we find%
\[
j_{z}^{\perp}(w)=\frac{1-\overline{\Theta_{J}(z)}\Theta_{J}(w)}{1-\bar{z}w}%
\]
and hence
\[
1-\delta_{J^{\perp}}^{2}\left(  z,w\right)  =\frac{\left(  1-\left\vert
\Theta_{J}\left(  z\right)  \right\vert ^{2}\right)  \left(  1-\left\vert
\Theta_{J}\left(  w\right)  \right\vert ^{2}\right)  }{\left\vert
1-\overline{\Theta_{J}(z)}\Theta_{J}(w)\right\vert ^{2}}\frac{(1-\left\vert
z\right\vert ^{2})(1-|w|^{2})}{\left\vert 1-\bar{z}w\right\vert ^{2}}.
\]
We can now use (\ref{magic}) on both fractions and continue with%
\[
1-\delta_{J^{\perp}}^{2}\left(  z,w\right)  =\frac{1-\rho^{2}(z,w)}{1-\rho
^{2}(\Theta_{J}(z),\Theta_{J}(w))}.
\]
Doing the algebra we obtain%
\[
\delta_{J^{\perp}}\left(  z,w\right)  =\sqrt{\frac{\rho^{2}(z,w)-\rho
^{2}(\Theta_{J}(z),\Theta_{J}(w))}{1-\rho^{2}(\Theta_{J}(z),\Theta_{J}(w))}}.
\]
In particular
\[
\delta_{J^{\perp}}\leq\rho=\delta_{\mathcal{H}_{1}}%
\]
with equality holding if and only if $\Theta_{J}(z)=\Theta_{J}(w).$

\subsection{Triples of Points and the Shape Invariant}

When we move away from the Hardy space computation becomes complicated.
Suppose $H$ is a RKHS on $X$ with kernel functions $\left\{  k_{x}\right\}  $
and associated distance function $\delta.$ Select distinct $x,y,z\in X.$ We
consider the invariant subspace $J$ of functions which vanish at $x,$ the
orthogonal complement of the span of $k_{x}.$ We will denote the kernel
functions for $J$ by $\left\{  j_{z}\right\}  .$ We want to compute
$\delta_{J}(y,z)$ in terms of other data.

For any $\zeta\in X,$ $j_{z}$ equals $k_{\zeta}$ minus the projection onto
$J^{\perp}$ of $k_{\zeta}.$ The space $J^{\perp}$ is one dimensional and
spanned by $k_{x}$ hence we can compute explicitly%
\[
j_{y}=k_{y}-\frac{k_{y}(x)}{\left\Vert k_{x}\right\Vert ^{2}}k_{x},
\]
and there is a similar formula for $j_{z}.$ We will need $\left\vert
\,\left\langle j_{y},j_{z}\right\rangle \right\vert ^{2}.$
\begin{align}
\left\vert \,\left\langle j_{y},j_{z}\right\rangle \right\vert ^{2}  &
=\left\vert k_{y}(z)-\frac{k_{y}(x)}{\left\Vert k_{x}\right\Vert ^{2}}%
k_{x}(z)\right\vert ^{2}\nonumber\\
&  =\left\vert k_{y}(z)\right\vert ^{2}+\frac{\left\vert k_{y}(x)k_{x}%
(z)\right\vert ^{2}}{\left\Vert k_{x}\right\Vert ^{4}}-2\operatorname{Re}%
k_{y}(z)\overline{\frac{k_{y}(x)}{\left\Vert k_{x}\right\Vert ^{2}}k_{x}(z)}.
\label{sofar}%
\end{align}
Recall that for any two distinct elements $\alpha,\beta$ of the set $\left\{
x,y,z\right\}  $ we have
\begin{equation}
\left\vert k_{\alpha}(\beta)\right\vert ^{2}=\left\vert k_{\beta}%
(\alpha)\right\vert ^{2}=\left\Vert k_{\alpha}\right\Vert ^{2}\left\Vert
k_{\beta}\right\Vert ^{2}\left(  1-\delta^{2}(\alpha,\beta)\right)  .
\label{repl}%
\end{equation}
We use this in (\ref{sofar}) to replace the quantities such as the ones on the
left in (\ref{repl}) with the one on the right. Also, we will write
$\delta_{\alpha\beta}^{2}$ rather than $\delta^{2}(\alpha,\beta).$ For the
same $\alpha,\beta$ we define $\theta_{\alpha\beta}$ and $\phi_{\alpha\beta}$
with $0$ $\leq\theta_{\alpha\beta}\leq\pi,$ and $\phi_{\alpha\beta}$ with $0$
$\leq\phi_{\alpha\beta}$\ $<\pi$ by%
\[
k_{\alpha}\left(  \beta\right)  =\left\langle k_{\alpha},k_{\beta
}\right\rangle =\left\Vert k_{\alpha}\right\Vert \left\Vert k_{\beta
}\right\Vert \left(  \cos\theta_{a\beta}\right)  e^{i\phi_{\alpha\beta}}%
\]
and we set
\[
\Upsilon=\cos\theta_{xy}\cos\theta_{yz}\cos\theta_{zx}\cos\left(  \phi
_{xy}+\phi_{yz}+\phi_{zx}\right)  .
\]
We now continue from (\ref{sofar}) with%
\begin{align*}
\left\vert \,\left\langle j_{y},j_{z}\right\rangle \right\vert ^{2}  &
=\left\Vert k_{z}\right\Vert ^{2}\left\Vert k_{y}\right\Vert ^{2}\left(
1-\delta_{zy}^{2}\right)  +\left\Vert k_{y}\right\Vert ^{2}\left\Vert
k_{z}\right\Vert ^{2}\left(  1-\delta_{xy}^{2}\right)  \left(  1-\delta
_{xz}^{2}\right)  -2\left\Vert k_{y}\right\Vert ^{2}\left\Vert k_{z}%
\right\Vert ^{2}\Upsilon\\
&  =\left\Vert k_{z}\right\Vert ^{2}\left\Vert k_{y}\right\Vert ^{2}\left\{
\left(  1-\delta_{zy}^{2}\right)  +\left(  1-\delta_{xy}^{2}\right)  \left(
1-\delta_{xz}^{2}\right)  -2\Upsilon\right\}  .
\end{align*}

Similar calculations give%
\begin{align*}
\left\Vert j_{y}\right\Vert ^{2}  &  =\left\Vert k_{y}\right\Vert ^{2}%
-\frac{\left\vert k_{x}(y)\right\vert ^{2}}{\left\Vert k_{x}\right\Vert ^{2}%
}=\left\Vert k_{y}\right\Vert ^{2}-\left\Vert k_{y}\right\Vert ^{2}\left(
1-\delta_{xy}^{2}\right)  =\left\Vert k_{y}\right\Vert ^{2}\delta_{xy}^{2}\\
\left\Vert j_{z}\right\Vert ^{2}  &  =\left\Vert k_{z}\right\Vert ^{2}%
\delta_{xz}^{2}%
\end{align*}
Hence%
\begin{align*}
\delta_{J}^{2}(y,z)  &  =1-\frac{\left\vert \,\left\langle j_{y}%
,j_{z}\right\rangle \right\vert ^{2}}{\left\Vert j_{y}\right\Vert
^{2}\left\Vert j_{z}\right\Vert ^{2}}\\
&  =1-\frac{\left\Vert k_{z}\right\Vert ^{2}\left\Vert k_{y}\right\Vert
^{2}\left\{  \left(  1-\delta_{zy}^{2}\right)  +\left(  1-\delta_{xy}%
^{2}\right)  \left(  1-\delta_{xz}^{2}\right)  -2\Upsilon\right\}
}{\left\Vert k_{y}\right\Vert ^{2}\delta_{xy}^{2}\left\Vert k_{z}\right\Vert
^{2}\delta_{xz}^{2}}\\
&  =\frac{\delta_{xy}^{2}\delta_{xz}^{2}-\left(  1-\delta_{zy}^{2}\right)
-\left(  1-\delta_{xy}^{2}\right)  \left(  1-\delta_{xz}^{2}\right)
+2\Upsilon}{\delta_{xy}^{2}\delta_{xz}^{2}}\\
&  =\frac{\delta_{xy}^{2}+\delta_{xz}^{2}+\delta_{zy}^{2}-2+2\Upsilon}%
{\delta_{xy}^{2}\delta_{xz}^{2}}.
\end{align*}
Thus we can write the very symmetric formula%
\begin{equation}
\delta(y,z)\delta_{J}(y,z)=\frac{\sqrt{\delta_{xy}^{2}+\delta_{xz}^{2}%
+\delta_{zy}^{2}-2+2\Upsilon}}{\delta_{xy}\delta_{xz}\delta_{yz}}.
\label{shape}%
\end{equation}

One reason for carrying this computation through is to note the appearance of
$\Upsilon.$ This quantity, which is determined by the ordered triple $\left\{
k_{x},k_{y},k_{z}\right\}  $ and which is invariant under cyclic permutation
of the three, is a classical invariant of projective and hyperbolic geometry.
It is called the shape invariant. The triple determines an ordered set of
three points in $P(H)$ the projective space over $H;$ $p_{x}=\left[
k_{x}\right]  ,$ $p_{y},$ and $p_{z}.$ Modulo some minor technicalities which
we omit, one can regard the three as vertices of an oriented triangle
$T_{xyz}.$ The edges of the triangle are the geodesics connecting the vertices
and the surface is formed by the collection of all geodesics connecting points
on the edges. In Euclidian space two triangles are congruent, one can be moved
to the other by an action of the natural isometry group, if and only if the
set of side lengths agree. That is not true in projective space. The correct
statement there is that two triangles are congruent if and only if the three
side lengths and the shape invariants match \cite{BT}, \cite{BR}. Using the
natural geometric structure of complex projective space one can also define
and compute the area of $T_{xyz}.$ It turns out, roughly, that once the side
lengths are fixed then $\Upsilon$ determines the area of $T_{xyz}$ and vice
versa \cite{HM}. Further discussion of the shape invariant is in \cite{Go} and
\cite{BS}.

The reason for mentioning all this is that $\Upsilon$ was the one new term
that appeared in (\ref{shape}) and it is slightly complicated. The fact that
this quantity has a life of its own in geometry suggests that perhaps the
computations we are doing are somewhat natural and may lead somewhere interesting.

In B\o e's work on interpolating sequences in RKHS with complete NP kernel
\cite{B} (see also \cite{Sa}) he makes computations similar in spirit and
detail to the ones above. Informally, he is extracting analytic information
from a geometric hypothesis. It is plausible that knowing how such techniques
could be extended from three points to $n$ points would allow substantial
extension of B\o e's results.

\subsection{Monotonicity Properties}

We saw that when $J$ was an invariant subspace of the Hardy space
$\mathcal{H}_{1}$ then for all $x,y$ in the disk%
\[
\delta_{J}(x,y)=\delta_{\mathcal{H}_{1}}(x,y)\geq\delta_{J^{\perp}}(x,y).
\]
It is not clear what, if any, general pattern or patterns this is an instance
of. Here we give some observations and computations related to that question.

\subsubsection{Maximal Multipliers}

Fix $H$ and $X.$ Recall Proposition \ref{multiplier modulus of continuity};
given $x,y\in X,$ $x\neq y$ we denoted by $G_{x,y}$ be the multiplier of $H$
of norm $1$ which has $G(x)=0$ and subject to those conditions maximizes
$\operatorname{Re}G(y).$ The Proposition stated that $\operatorname{Re}%
G_{x,y}\left(  y\right)  \leq\delta_{H}(x,y)$ and that equality sometimes
held. If equality does hold we will say that $x,y$ have an maximal multiplier
and we will call $G_{x,y}$ the maximal multiplier.

\begin{proposition}
If $x,y\in X$ have a maximal multiplier and $J$ is any closed multiplier
invariant subspace of $H$ then $\delta_{J}(x,y)\geq\delta_{H}(x,y).$

\begin{proof}
Because $J$ is a closed multiplier invariant subspace of $H,$ the maximal
multiplier $G_{x,y}$ is also a multiplier of $J$ and has a norm, as a
multiplier on $J,$ at most one. Thus $G_{x,y}$ is a competitor in the extremal
problem associated with applying Proposition
\ref{multiplier modulus of continuity} to $J.$ Hence, by that proposition we
have $\operatorname{Re}G_{x,y}\left(  y\right)  \leq\delta_{J}(x,y).$ On the
other hand our hypothesis is that $\operatorname{Re}G_{x,y}\left(  y\right)
=\delta_{H}(x,y).$
\end{proof}
\end{proposition}

If $H$ has a complete NP kernel then every pair of points, $x,y,$ has a
maximal multiplier. Hence

\begin{corollary}
If $H$ has a complete NP kernel and $J$ is any closed multiplier invariant
subspace of $H$ then for all $x,y\in X,$ $\delta_{J}(x,y)\geq\delta_{H}(x,y).
$
\end{corollary}

The converse of the corollary is not true. Having a complete NP kernel is not
a necessary condition in order for every pair of points to have an maximal
multiplier; it is sufficient that the kernel have the scalar two point Pick
property, see \cite[Ch. 6,7]{AM}.

\subsubsection{Spaces with Complete Nevanlinna Pick Kernels}

Suppose that $H$ is a RKHS on $X$ with a complete NP kernel $K(\cdot,\cdot).$
Suppose also, and this is for convenience, that we have a distinguished point
$\omega\in X$ such that $\forall x\in X,$ $K(\omega,x)=K(x,\omega)=1.$ The
following information about invariant subspaces of $H$ is due to McCullough
and Trent \cite{MT}, further information is in \cite{GRS}.

\begin{proposition}
Suppose $J$ is a closed multiplier invariant subspace of $H.$ There are
multipliers $\left\{  m_{i}\right\}  $ so that the reproducing kernel for $J$
is%
\begin{equation}
K_{J}(x,y)=\left(  \sum m_{i}(x)\overline{m_{i}(y)}\right)  K(x,y). \label{mt}%
\end{equation}

\end{proposition}

\begin{corollary}
\label{np}If $H$ has a complete NP kernel and $J$ is any closed multiplier
invariant subspace of $H$ then for all $x,y\in X$
\[
\delta_{J}(x,y)\geq\delta_{H}(x,y)\geq\delta_{J^{\perp}}(x,y).
\]
{}
\end{corollary}

\begin{proof}
We start with formula (\ref{mt}) which we rewrite for convenience as
\begin{equation}
K_{J}(x,y)=A(x,y)K(x,y). \label{AK}%
\end{equation}
The first inequality is the statement of the previous corollary. Alternatively
we could start from the previous equality and use (\ref{inequality}) to
compare $\delta_{H}$ and $\delta_{J}$ yielding a quantitative version of the
desired inequality.

For the second inequality first note that $K_{J^{\perp}}=K-K_{J}=K-AK=\left(
1-A\right)  K.$

(Note that for any $x$, $A(x,x)\leq1$ because it is the ratio of the squared
norms of two kernel functions and the one on top, being a projection of the
one on bottom, has smaller norm. Also, by Cauchy-Schwarz, $\left\vert
A(x,y)\right\vert ^{2}\leq A(x,x)A(y,y)\leq1.$ To rule out the case of
equality note that if $A(x,x)=1$ then $k_{x}\in J$ and hence every function in
$J^{\perp}$ vanishes at $x$ which puts $x$ outside the domain of
$\delta_{J^{\perp}}.$)

Recalling the formula for $\delta$ we see that our claim will be established
if we can show for the $x,y\in X$ that are covered by the claim we have
\[
\frac{\left\vert 1-A(x,y)\right\vert ^{2}}{\left(  1-A(x,x)\right)  \left(
1-A(y,y\right)  )}\geq1.
\]
We have
\begin{align*}
\left\vert 1-A(x,y)\right\vert ^{2}  &  \geq(1-\left\vert A(x,y)\right\vert
)^{2}\\
&  \geq\left(  1-A(x,x)^{1/2}A(y,y)^{1/2}\right)  ^{2}\\
&  =1-2A(x,x)^{1/2}A(y,y)^{1/2}+A(x,x)A(y,y)\\
&  \geq1-A(x,x)-A(y,y)+A(x,x)A(y,y)\\
&  =\left(  1-A(x,x)\right)  \left(  1-A(y,y\right)  ).
\end{align*}
Here the passage from the first line to the second uses Cauchy-Schwarz, the
passage from third to fourth uses the arithmetic mean, geometric mean inequality.
\end{proof}

\subsubsection{Inequalities in the Other Direction; Bergman Type Spaces}

Let $H=\mathcal{H}_{2}$ be the Bergman space. That is, $H$ is the RKHS of
holomorphic functions on the disk with reproducing kernel $K(z,w)=(1-\bar
{w}z)^{-2}.$ Let $J$ be the invariant subspace consisting of all functions in
$H $ which vanish at the origin.

\begin{proposition}
For all $z,w\in\mathbb{D}$%
\[
\delta_{J}(z,w)\leq\delta_{H}(z,w).
\]

\begin{proof}
We have%
\begin{align*}
K_{J}(z,w)  &  =K(z,w)-1=\frac{1}{(1-\bar{w}z)^{2}}-1\\
&  =\frac{\bar{w}z(2-\bar{w}z)}{(1-\bar{w}z)^{2}}=\frac{2\bar{w}z(1-\frac
{\bar{w}z}{2})}{(1-\bar{w}z)^{2}}\\
&  =B(z.w)K(z,w).
\end{align*}
At this stage we can see the difference between this situation and the one in
the previous section. Here the ratio $K_{J}/K$ is \textit{not }a positive
definite function. To finish we need to show
\[
\frac{\left\vert B(z,w)\right\vert ^{2}}{B(z,z)B(w,w)}\geq1.
\]
Thus we need to show
\[
\left\vert 2-\bar{w}z\right\vert ^{2}\geq(2-\bar{w}w)(2-\bar{z}z).
\]
Equivalently, we need
\[
-4\operatorname{Re}\bar{w}z\geq-2\bar{w}w-2\bar{z}z.
\]
This follows from the inequality between the arithmetic and geometric means.
\end{proof}
\end{proposition}

In fact this example is just the simplest case of a general pattern introduced
in \cite{HJS} and \cite{MR}. In \cite{MR} McCullough and Richter introduce a
general class of RKHS which share many of the properties of the Bergman space.
In particular their work covers the spaces $\mathcal{H}_{\alpha},$
$1\leq\alpha\leq2,$ and we will focus on that case. Suppose $V$ is in
invariant subspace of some $\mathcal{H}_{\alpha},$ $1\leq\alpha\leq2$ and that
$V$ has index 1, that is $\dim V\ominus zV=1.\ $Let $\left\{  k_{z}\right\}  $
be the reproducing kernels for $\mathcal{H}_{\alpha}$ and $\left\{
j_{z}\right\}  $ be those for $V.$

\begin{proposition}
[Corollary 0.8 of \cite{MR}]There is a function $G\in\mathcal{H}_{\alpha}$ and
a positive semidefinite sesquianalytic function $A(z,w)$ so that for
$z,w\in\mathbb{D}$%
\[
j_{z}(w)=\overline{G(z)}G(w)(1-\bar{z}wA(z,w))k_{z}(w).
\]

\end{proposition}

\begin{corollary}%
\[
\delta_{V}(z,w)\leq\delta_{\mathcal{H}_{\alpha}}(z,w).
\]

\begin{proof}
The factors of $G$ do not affect $\delta.$ After they are dropped the argument
is then the same as in the proof of Corollary \ref{np}.%
\begin{align*}
\left\vert 1-\bar{z}wA(z,w)\right\vert ^{2}  &  \geq\left(  1-\left\vert
\bar{z}wA(z,w)\right\vert \right)  ^{2}\\
&  \geq(1-\left\vert z\right\vert A(z,z)^{1/2}\left\vert w\right\vert
A(w,w)^{1/2})^{2}\\
&  \geq1-2\left\vert z\right\vert A(z,z)^{1/2}\left\vert w\right\vert
A(w,w)^{1/2}+\left\vert z\right\vert ^{2}A(z,z)^{1/2}\left\vert w\right\vert
^{2}A(w,w)\\
&  \geq1-\left\vert z\right\vert ^{2}A(z,z)-\left\vert w\right\vert
^{2}A(w,w)+\left\vert z\right\vert ^{2}A(z,z)\left\vert w\right\vert
^{2}A(w,w)\\
&  \geq(1-\left\vert z\right\vert ^{2}A(z,z))(1-\left\vert w\right\vert
^{2}A(w,w))
\end{align*}
which is what is needed.

Again, the passage from the first line to the second uses Cauchy-Schwarz, the
passage from third to fourth uses the arithmetic mean, geometric mean inequality.
\end{proof}
\end{corollary}

It is not clear in this context what happens with spaces of the form
$J^{\perp}.$ The following computational example suggests the story may be
complicated. Let $J$ be the invariant subspace of $H=\mathcal{H}_{2}$
consisting of functions $f$ with $f\left(  0\right)  =f^{\prime}\left(
0\right)  =0$. The reproducing kernel for $J^{\perp}$ is
\[
K_{J^{\perp}}(z,w)=1+2\bar{w}z.
\]
To compare $\delta_{J^{\perp}}$ with $\delta_{H}$ we compare the quantities
$1-\delta^{2}.$ If we are looking for an example where the results for spaces
with complete NP kernels fails then we need to find $z,w$ such that
\[
\frac{\left\vert 1+2\bar{w}z\right\vert ^{2}}{\left(  1+2\left\vert
w\right\vert ^{2}\right)  \left(  1+2\left\vert z\right\vert ^{2}\right)
}\leq\frac{\left(  1-\left\vert w\right\vert ^{2}\right)  ^{2}\left(
1-\left\vert z\right\vert ^{2}\right)  ^{2}}{\left\vert 1-\bar{w}z\right\vert
^{4}}.
\]
This certainly does not hold if either $\left\vert z\right\vert $ or
$\left\vert w\right\vert $ is close to $1$ however things are different near
the origin. Suppose $z=-w=t>0.$ The inequality we want is
\[
\frac{\left(  1-2t^{2}\right)  ^{2}}{\left(  1+2t^{2}\right)  ^{2}}\leq
\frac{\left(  1-t^{2}\right)  ^{4}}{\left(  1+t^{2}\right)  ^{4}}.
\]
This fails if $t$ is close to one but we study it now for $t$ near $0.$ In
that case the left hand side is $1-8t^{2}+32t^{4}-96t^{6}+O\left(
t^{8}\right)  $ and the other side is $1-8t^{2}+32t^{4}-88t^{6}+O\left(
t^{8}\right)  $. Hence for small $t$ the inequality does hold.

\begin{proposition}
With $H$ and $J$ as described, if $t$ is small and positive then%
\[
\delta_{J^{\perp}}(t,-t)>\delta_{H}(t,-t).
\]
For $t$ near $1$ the inequality is reversed.
\end{proposition}

\section{Questions}

We mentioned in the introduction that most questions in this area have not
been studied. Here we mention a few specific questions which had our interest
while preparing this paper and which indicate how little is know.

1. Suppose $H$ and $H^{\prime}$ are two RKHSs on the same $X$ with distance
functions $\delta$ and $\delta^{\prime};$ and suppose further, in fact, that
$H$ and $H^{\prime}$ are the same spaces of functions with equivalent norms.
What conclusions follow about $\delta$ and $\delta^{\prime}?$

2. In the other direction, what conclusion can one draw about the relation
between $H$ and $H^{\prime}$ if the identity map from $\left(  X,\delta
\right)  $ to $\left(  X,\delta^{\prime}\right)  $ is, say, a contraction or
is bilipschitz?

3. It seams plausible that there is a more complete story to be told related
to Corollary \ref{np}. What is the full class of RKHS for which those
conclusions hold? What assumptions beyond those conclusions are needed to
insure that the space being considered has a complete NP kernel?

4. Given $X$ what metrics $\delta$ can arise from a RKHS $H$ on $X.$ The
question is extremely broad but notice that if you assume further that $H$
must have a complete NP kernel then, by virtue of the realization theorem, a
necessary and sufficient condition is that for some $n$ there is an isometric
map of $\left(  X,\delta\right)  $ into $\left(  \mathbb{B}^{n},\rho\right)
.$ Although that answer is perhaps not particularly intuitive it does give a
condition that is purely geometric, none of the Hilbert space discussion survives.

\end{document}